\documentclass[10pt, a4paper]{amsart}

\usepackage{a4, a4wide}
\usepackage{amssymb}
\usepackage{amsmath, bm}
\usepackage{amsthm}
\usepackage{amstext}
\usepackage{amscd}
\usepackage{latexsym}
\usepackage{graphics}
\usepackage{color}
\usepackage{enumitem}
\usepackage[all, cmtip]{xy}
\usepackage{enumitem}
\usepackage{comment}
\usepackage[colorlinks, pagebackref]{hyperref}
\hypersetup{
  colorlinks=true,
  citecolor=blue,
  linkcolor=blue,
  urlcolor=blue}

\usepackage{cleveref}

\makeatletter
\def\@tocline#1#2#3#4#5#6#7{\relax
  \ifnum #1>\c@tocdepth 
  \else
    \par \addpenalty\@secpenalty\addvspace{#2}%
    \begingroup \hyphenpenalty\@M
    \@ifempty{#4}{%
      \@tempdima\csname r@tocindent\number#1\endcsname\relax
    }{%
      \@tempdima#4\relax
    }%
    \parindent\z@ \leftskip#3\relax \advance\leftskip\@tempdima\relax
    \rightskip\@pnumwidth plus4em \parfillskip-\@pnumwidth
    #5\leavevmode\hskip-\@tempdima
      \ifcase #1
      \or\or \hskip 2em \or \hskip 2homologyem \else \hskip 3em \fi%
      #6\nobreak\relax
    \dotfill\hbox to\@pnumwidth{\@tocpagenum{#7}}\par
    \nobreak
    \endgroup
  \fi}
\makeatother

\theoremstyle{plain}

\newtheorem{theorem}{Theorem}[section]
\newtheorem{question}[theorem]{Question}
\newtheorem{lemma}[theorem]{Lemma}
\newtheorem{corollary}[theorem]{Corollary}
\newtheorem{proposition}[theorem]{Proposition}
\newtheorem{conjecture}[theorem]{Conjecture}
\theoremstyle{definition}

\newtheorem{remark}[theorem]{Remark}

\newtheorem{definition}[theorem]{Definition}

\numberwithin{equation}{section}

\newcommand{\dlim}{\mathop{\varinjlim}\limits}  
\DeclareMathOperator{\hlim}{{\rm{holim}}}

\newcommand{\Sm}{{\rm Sm}}
\newcommand{\Spc}{{\rm Spc}}

\newcommand{\CH}{{\rm CH}}
\newcommand{\E}{{\rm E}}
\newcommand{\K}{{\rm K}}
\newcommand{\coh}{\operatorname{H}}

\newcommand{\Ker}{{\rm Ker \ }}
\newcommand{\Coker}{{\rm Coker \ }}

\newcommand{\Hom}{{\rm Hom}}

\newcommand{\im}{{\rm Im \ }}
\newcommand{\Spec}{{\rm Spec \,}}

\newcommand{\GL}{{\rm GL}} 
   
\newcommand{\SL}{{\rm SL}}
\newcommand{\BSL}{{\rm BSL}}
\newcommand{\Sp}{{\rm Sp}} 
\newcommand{\BSp}{{\rm BSp}}

\newcommand{\ie}{{\it i.e.\/},\ }

\newcommand{\nis}{Nis}

\newcommand{\KM}{{\&K}^{\rm M}}                      
\newcommand{\KMW}{{\&K}^{\rm MW}}		
\newcommand{\piA}{{\bm \pi}^{\A^1}}   		

\newcommand{\A}{\mathbb A}

\newcommand{\G}{\mathbb G}
\newcommand{\F}{\mathbb F}

\newcommand{\w}{\omega}

\newcommand{\sO}{\mathcal O}

\newcommand{\simps}{{\rm{\Delta^{op}Shv}}}

\newcommand{\Rmap}{{\rm RMap}}
\newcommand{\SK}{{\rm SK}}

\def\<{\langle}
\def\>{\rangle} 
\def\-{\overline} 
\def\~{\widetilde}
\def\^{\widehat}
\def\fr{\mathfrak}
\def\@{\mathcal}
\def\#{\mathbb}
\def\&{\mathbf}
\def\_{\underline}
\def\Dot{\bullet} 
\def\x{\times}

\newcommand{\fixme}[1]{\textcolor{red}{#1}}

\input{xy}
\xyoption{all}

\begin{document}

\title[Cancellation and splitting  of Symplectic modules and Euler class group]{Cancellation and splitting  of Symplectic modules in the critical range and Euler class group
}

\author{Rakesh Pawar}
\address{Department of Mathematics, Indian Institute of Technology Jodhpur, Jodhpur, Rajasthan, 342030, India
}
\email{pawarrakesh.math@gmail.com}
\author{Husney Parvez Sarwar}
\address{Department of Mathematics, Indian Institute of Technology Kharagpur,
Kharagpur 721302, West Bengal, India}
\email{ parvez@maths.iitkgp.ac.in \\
 mathparvez@gmail.com}

\date{\today}

\date{\today}


\begin{abstract} 
In this paper, we discuss the cancellation and splitting of the symplectic modules. The symplectic cancellation result presented here can be thought of as an analog of the  projective module cancellation result of Fasel.
The symplectic splitting is similar to Murthy's splitting theorem.  
To prove the cancellation and splitting, we analyze the Postnikov towers in the
$\mathbb{A}^1-$homotopy category.
We study the top cohomology group with coefficients in certain $\#A^1-$homotopy sheaf of even quadrics and show certain structural properties of these cohomology groups.
As another application of this, we answer partially a question of Mrinal Das about
the isomorphism of $(d-1)-$th Euler class group and  $(d-1)-$th Chow group of a smooth affine scheme of dimension $d$.
\end{abstract}

\maketitle
\tableofcontents
\section{Introduction}
Let us begin with two classical results. Let $R$ be a commutative Noetherian ring of dimension $d$ and let $P$ be a finitely generated projective $R-$module of rank $r$. 
For $r>d$, Serre's theorem (\cite{Serre58}) says that $P\cong Q\oplus R$ ({\it Serre's splitting theorem}) and Bass' theorem (\cite{Bass64}) says that $P\oplus Q \cong P'\oplus Q$ implies that $ P\cong P'$ ({\it Bass's cancellation theorem}). For ${\rm rank} (P)=\dim R$, the non-trivial tangent bundle over the real algebraic sphere which is stably free, asserts that both the above theorems are false in the case $r=d=2$. However, over an algebraically closed field, Murthy \cite{Mu94} and Suslin \cite{Sus77} proved the following two results on splitting and cancellation respectively. Let $R$ be a smooth affine algebra of Krull dimension $d$ over an algebraically closed field and $P$ a finitely generated projective $R-$module with ${\rm rank}(P)=d$. M. P. Murthy \cite[Theorem 3.8]{Mu94} proved that $c_d(P)=0$ if and only if $P\cong Q \oplus R$ where $c_d(P)$ is the top Chern class of $P$ in the Chow group $\CH^d(\Spec(R))$. Suslin \cite{Sus77} proved that 
$P\oplus Q \cong P'\oplus Q$ implies that $P\cong P'$ even if $R$ is not a smooth algebra. Recently,
Fasel \cite{Fas24} proved that (under the assumption  $\frac{1}{d!}\in R\setminus 0$) if ${\rm rank}(P)=d-1$,
then also the cancellation holds, \ie $P\oplus Q \cong P'\oplus Q$ implies that $ P\cong P'$ which was Suslin's conjecture  \cite[Conjecture 1]{Fas24}.
Fasel's result is also optimal in the view of Mohan Kumar's example \cite{MK85}. Similar to Suslin's cancellation conjecture, there is a conjecture for the splitting called Murthy's splitting conjecture \cite{AF15}.
\begin{conjecture}[Murthy's splitting conjecture \cite{AF15}]
 Let $R$ be a smooth affine algebra of Krull dimension $d\geq 3$ over an algebraically closed field. Let $P$ be a projective $R-$module of {\rm rank} $d-1$. Then the top Chern class $c_{d-1}(P)=0$ in $\CH^{d-1}(\Spec(R))$ if and only if $P\cong Q \oplus R$ for some $R-$module $Q$.
\end{conjecture}

Murthy's splitting conjecture remains open in general. However when $d=3,4$, the conjecture has been solved affirmatively by Asok--Fasel (see \cite{AF15}, \cite{AF14}) and recently it is resolved in the positive when the field $k$ is of characteristic $0$ by combining the works of \cite{AF15}, \cite{AF14} and \cite{ABH23}. 

In this paper the objective is two-fold. In the first, we study the symplectic modules \ie projective modules with a symplectic form (see~\Cref{def:symp-mod}) and offer the following two results which are analogous to  Fasel's theorem and  Murthy's splitting conjecture in the context of symplectic modules. Secondly, we address a question of Das \cite[Question 7.1]{MD21} which connects Euler class groups with Chow groups.

Let  $\#H(R)$ denote the symplectic hyperbolic plane over a ring $R$ defined by rank 2 free $R$-module $R\oplus R$ endowed with the bilinear form associated to the $2\x 2$-matrix $\left(\begin{smallmatrix}0 & 1 \\ -1 & 0 \end{smallmatrix}\right)$.
Also, let us denote by $V_{2m}^{\Sp}(X)$ the set of isomorphism classes of symplectic $R-$modules of rank $2m$ for any $m\geq 1$, where $X=\Spec R$.
\begin{theorem}(Theorem \ref{cancel2n-2})
({\it Corank 2 symplectic cancellation})
     Let $X=\Spec R$ be a smooth affine scheme of dimension $2n$ ($n\geq 1$) over an algebraically closed field $k$. Let $(2n)!$ be invertible in $k.$ Then the map 
    $$V_{2n-2}^{\Sp}(X)\to V_{2n}^{\Sp}(X)$$ sending $[P]\mapsto [P\perp \#H(R)]$ is injective. More explicitly, if two symplectic rank $(2n-2)-$projective modules $E$ and $E'$ are stably isomorphic \ie $E \oplus \#H(R)\simeq E'\oplus \#H(R)$, then  $E \simeq E'$. 
\end{theorem}

\begin{theorem} (Theorem \ref{thm:splitting})
({\it Lower Corank symplectic splitting})
Let $R$ be a smooth affine scheme over a field $k$ and $X=\Spec (R)$.  Let $P$ be a symplectic $R-$module of rank $2n$ ($n\geq 1$).

$(1)$ Assume that $\dim(R)=2n$. Then $P\cong Q \oplus \#H(R) $ if and only if the Euler class $e_{2n}(P)=0$ in $\widetilde{\CH}^{2n}(X)$.

$(2)$ Assume that $\dim(R)=2n+1$, and $\coh^{2n+1}(X, \piA_{2n}(\mathbb{A}^{2n}\setminus\{0\}))=0.$ Then $P\cong Q \oplus \#H(R) $ if and only if the Euler class $e_{2n}(P)=0$ in $\widetilde{\CH}^{2n}(X)$.
\end{theorem}

The cancellation of symplectic modules has been studied by Bass, Bhatwadekar \cite{Bh01} and 
Du \cite{PD22}. The approach we take throughout this paper is to exploit the obstruction theoretic machinery developed by Morel-Voevodsky~\cite{MV99} and Morel~\cite[Appendix B]{Mor12} which reduces the lifting problems to understanding Nisnevich cohomology groups of homotopy sheaves. This strategy is successfully deployed by Asok-Fasel~ \cite{AF15}, \cite{AF14} and also Du \cite{PD22} and Asok-Bachmann-Hopkins~\cite{ABH23} among others in a series of works to study the classification/splitting/enumeration problems for projective modules on smooth affine algebras over a field. In this paper, we use the vanishing theorem of Fasel~\cite{Fas24} at crucial places and in addition to that use the following result that lists the conditions under which the top cohomology groups vanish of certain homotopy sheaves of quadrics on smooth affine variety. We believe that the following result is of independent interest and would serve a purpose in other obstruction theoretic arguments.

For a field $k$, let us recall the affine quadrics (for each $n\geq 1$), an odd quadric $$Q_{2n-1}=\Spec \dfrac{k[x_1, y_1, \dots, x_{n}, y_n]}{\sum_{i=1}^{n} x_iy_i-1} $$
and even quadric $$Q_{2n}=\Spec \dfrac{k[x_1, y_1, \dots, x_{n}, y_n, z]}{\sum_{i=1}^{n} x_iy_i-z(1-z)}.$$

We consider the 2nd non-trivial $\#A^1-$homotopy sheaf of the even quadrics $Q_{2d-2}$ (for $d\geq 2$) and consider the top cohomology of the sheaves $\piA_d({Q_{2d-2}})$ on a smooth affine $k-$scheme $X$ of dimension $d$.  
\begin{theorem} \label{vanishing-intro} (Theorem \ref{vanishing})
Let $X$ be a smooth affine scheme of dimension $d$ over an algebraically closed field $k$ of characteristic $\neq 2$. Then
\begin{enumerate}
\item[$(1)$] if $d=2$,  $\coh^2(X, \piA_2(Q_{2}))=0.$ 

\item[$(2)$] if $d=3$, $\coh^3(X, \piA_3(Q_{4}))=0$ provided $k=\-{\F}_p$ (prime $p\neq 2$). 

\item[$(3)$] if $d=4$, $\coh^4(X, \piA_4(Q_{6}))=0.$

\item[$(4)$] if $d\geq 5$, $\coh^d(X,\piA_d(Q_{2d-2}))=0,$ provided $char(k)=0$.
\item[$(5)$] if $d\geq 5$, then $\coh^{d}(X, \piA_{d}(Q_{2d-2}))$ is $d!\cdot (d-1)!\cdot (d-2)!-$torsion.
\item[$(6)$] if $d\geq 5$ and $d!\in k^{\x}$, then $\coh^{d}(X, \piA_{d}(Q_{2d-2}))$ is $ (d-1)!\cdot (d-2)!-$torsion.
\item[$(7)$] for $d\geq 5$, $\coh^d(X,\piA_d(Q_{2d-2}))=0,$ provided Conjecture~\ref{AF-conj} (due to Asok-Fasel) holds.
\end{enumerate}

\end{theorem}

\begin{remark} The following are some comments about the above results.
\noindent
\begin{enumerate}
    \item The case $d=2$ complements the result proved in \cite[Theorem 3]{Fas24} as one notices that $\piA_2(Q_2)\simeq  \piA_2(\mathbb{A}^2\setminus \{0\})$.
 
 \item Conjecture \ref{AF-conj} roughly says that the Nisnevich sheaf $\piA_{d-1}(\A^{d-1}\backslash \{0\})$ is sandwiched between two known sheaves $\KM_{d+1}/24$ and $\&{GW}^{d-1}_d$. 

    \end{enumerate}
\end{remark}

 Perhaps unsurprisingly, the above~Theorem~\ref{vanishing-intro} has application in addressing the following question posed by Mrinal Das \cite{MD21} about the comparison map between the Euler class group (see for definition, Section \ref{Euler-Chow}, second paragraph) and the Chow group (\cite{Fulton}).
 
 \begin{question} (\cite[Question 7.1]{MD21})\label{Q-MD}  Let $p\neq 2$ be a prime and let $k=\-{\F}_p$. Let $X=\Spec R$ be a smooth affine scheme of dimension $d\geq 4$.
Is the natural map $$\psi: \E^{d-1}(R)\to \CH^{d-1}(X)$$ an isomorphism?
\end{question}
We note that the question was answered in the case $d=3$ in \cite[Theorem 7.2]{MD21}.
We have the following result which associates the positive answer to Question~\ref{Q-MD} with the vanishing of certain  cohomology groups.

\begin{theorem} (Corollary \ref{cor:iso})
Let $k$ be an algebraically closed field of characteristic $\neq 2$.
Let $X=\Spec A$ be a smooth affine $k-$scheme of dimension $d\geq 4$.

Consider the map $$\psi: \E^{d-1}(A)\to \CH^{d-1}(X).$$ Then the following statements hold. 

\begin{enumerate}
    \item[$(1)$] The map $\psi$ is always surjective.

 \item[$(2)$] If $d=4$, then the map $\psi$ is an isomorphism. 
 \item[$(3)$] If further characteristic of $k$ is $0$, then the map $\psi$ is an isomorphism. 
\item[$(4)$]  If $d\geq 5$, then the kernel of the map $\psi$ is $d!\cdot (d-1)!\cdot(d-2)!-$torsion.
         \item[$(5)$] If $d\geq 5$ and $d!\in k^{\x}$, then the kernel of the map $\psi$ is $(d-1)!\cdot(d-2)!-$torsion.
\item[$(6)$] Further for $d\geq 5$, if Conjecture \ref{AF-conj} holds, (which is the case when characteristic of $k=0)$  then the map $\psi$ is an isomorphism.
\end{enumerate}
\end{theorem}

\subsection{Acknowledgement}
The authors thank the referee for carefully going through the paper and for suggestions/corrections which improved the readability of the paper.
We thank Jean Fasel for his comments on the earlier draft of the paper.
The author RP was supported by UMPA, ENS de Lyon, and was supported by the French ANR project: Motivic homotopy, quadratic invariants and diagonal classes (ANR-21-CE40-0015) for major part of the work and also thanks Harish-Chandra Research Institute, Prayagraj and Kerala School of Mathematics, Kerala for support during the conclusion of the work. The author HPS would like to thank the Isaac Newton Institute for Mathematical Sciences, Cambridge, for support and hospitality during the programme [K-theory, algebraic cycles and motivic homotopy theory (KAH2) (2022)] where some parts of the work was carried out. This work was supported by EPSRC grant no EP/R014604/1. HPS also thanks N.B.H.M., Govt. of India for the grant No.02011/22/2023NBHM(R.P)/R $\And$ D II/, Dt.09-05-2023.

\section{Preliminaries}
In this section, we recollect some basic notions and results that will form the framework for the results proved in this paper. We refer the reader to 
\cite{Mor12}, \cite{MV99} and \cite{AF15}
for further details.

\subsection{Motivic homotopy theory}
First, we recall some basic notions in the homotopy theory.

Let $S$ be a Noetherian scheme of finite Krull dimension. 
$\Sm/S$ denotes the category of finite type smooth $S-$schemes. Let $(\Sm/S)_{\nis}$ denote the site with the Nisnevich topology on the category $\Sm/S$. The category of simplicial Nisnevich sheaves $\Spc_S:=\simps(\Sm/S)_{\nis}$ is a simplicial model category with the model structure as in~\cite[Theorem 1.4]{MV99}. By \cite[section 3.2, page 105]{MV99}, the above simplicial model category   $\Spc^{\A^1}_S:=\simps(\Sm/S)_{\nis}$ is endowed with the $\A^1-$model category structure. An object of $\Spc^{\A^1}_S$ is simply referred to as a \textit{Space}. Inverting the $\A^1-$weak equivalences in $\Spc^{\A^1}_S$, the resulting homotopy category is the homotopy category of smooth schemes over $S$ denoted by $\@H^{\A^1}(S).$ 

For the current paper, we take $S=\Spec k$ for a perfect field $k$ (unless stated otherwise) and denote $\@H^{\A^1}(k):= \@H^{\A^1}(\Spec k)$ for the $\A^1-$homotopy category of the field $k.$ For any object $\@X, \@Y\in \Spc^{\A^1}_k$, we denote $[\@X, \@Y]_{\A^1}:=\Hom_{\@H^{\A^1}(k)}(\@X, \@Y)$. For a smooth $k-$scheme $X$, we denote by $X$ itself the corresponding representable object $\Hom_{\Sm/k}(-, X)$ in $\Spc^{\A^1}_k$. Similarly, we can define the category of pointed spaces $Spc^{\A^1}_{\Dot, k}$, and the resulting pointed  $\A^1-$homotopy category $\@H^{\A^1}_{\Dot}(k)$. The homotopy category has bi-graded spheres that are obtained by the smash product of various copies of the simplicial circle $S^1$ and the algebraic circle $\#G_m=\#A^1\setminus\{0\}$ pointed at 1.
\medskip
\paragraph{\bf Conventions:} Throughout the paper, the sheaves of abelian groups $M$ are with respect to the Nisnevich topology and the associated cohomology groups $\coh^*(X, M)$ are the sheaf cohomology groups with respect to the Nisnevich topology on the corresponding scheme $X$.

\begin{definition}(\cite[Page 110]{MV99}) Let $(\@X, x)$ be a pointed space.
  For $n\geq 1$, $\piA_n(\@X, x)$ denotes the $n^{th}$ $\A^1-$homotopy sheaf defined as the Nisnevich sheafification of the presheaf $$U\mapsto [S^n\wedge U_+, (\@X, x)]_{\A^1}$$ where $U\in \Sm/k$. 
\end{definition}    
\begin{remark}
If the base field $k$ is perfect, for a pointed space $\@X$,  $\piA_1(\@X)$ is a sheaf of groups that is strongly $\A^1-$invariant and for $n\geq 2$, $\piA_n(\@X)$ is a strictly $\A^1-$invariant sheaf of abelian groups (see~\cite[Theorem 5.1, Theorem 4.47]{Mor12}).
\end{remark}

\subsection{Fiber sequences}\label{sec:fib-seq}
In this section, we recall the basics of fiber sequences and notions that will be used in the \Cref{sec:cancel-and-split}. For a thorough and detailed exposition we refer the reader to~\cite[subsection 2.1]{Fas24}.
To make the exposition self-contained we recall some relevant details.  

Let \begin{equation}\label{eq:fiberseq}
(F, f)\xrightarrow{p} (E, e)\xrightarrow{q} (B, b)
\end{equation} be a fiber sequence of pointed spaces. 
Here we suppose that $B$ is a fibrant space and $q$ is a fibration. 

Given a pointed space $\@X$, we have an exact sequence 
\begin{equation}\label{eq:long-exact-fiber}
    [\@X, \Omega(E, e)]_{\A^1} \xrightarrow{\Delta} [\@X, \Omega(B, b)]_{\A^1} \to [\@X, (F, f)]_{\A^1} \xrightarrow{p_*} [\@X, (E, e)]_{\A^1} \xrightarrow{q_*} [\@X, (B, b)]_{\A^1} 
    \end{equation} 
of pointed sets pointed at the base point $\@X\to *\xrightarrow{f} F$.
Here $\Omega(\@Y, y)$ is the loop space of a pointed space $(\@Y, y)$.
The exactness at $[\@X, (F, f)]_{\A^1}$ implies that the subgroup $\im(\Delta)$ of $[\@X, \Omega(B, b)]_{\A^1}$ is the stabilizer of the base point in $[\@X, (F, f)]_{\A^1}$. Thus, the inverse image under the map $p_*$ of the base point $\ast\in [\@X, (E, e)]_{\A^1} $ is given by the bijection of pointed sets $$(p_*)^{-1}(*)\simeq\dfrac{[\@X, \Omega(B, b)]_{\A^1}}{\im(\Delta)}$$ 
 where the quotient set is the set of orbits under the subgroup $\im (\Delta)$.
 The point to note is that the above description is for the inverse image of the base point $\@X\to *\xrightarrow{f} F$, but this doesn't describe $(p_*)^{-1}(\alpha)$ for any $\alpha\in [\@X, (F, f)]_{\A^1}.$  

However, if we assume that $(E, e)$ is an abelian group object in $\@H^{\A^1}_{\Dot}(k)$, then one can describe the inverse image $(p_*)^{-1}(\alpha)$ for any $\alpha\in [\@X, (F, f)]_{\A^1}.$ 
Here $(E, e)$ is an abelian group object \ie there exists a binary operation $$m: (E, e)\x (E, e)\to (E, e)$$ and an inverse operation $$i:(E, e)\to (E, e) $$ satisfying standard properties where the commutativity of diagrams in those properties hold in $\@H^{\A^1}_{\Dot}(k)$. 

Applying $\Rmap_*(\@X,-)$ to the fiber sequence, we have a fiber sequence of simplicial sets 
$$\Rmap_*(\@X, (F, f))\xrightarrow{p_*} \Rmap_*(\@X, (E, e))\xrightarrow{q_*} \Rmap_*(\@X, (B, b)).$$
Here $\Rmap_*(-, -)$ is a derived (pointed) mapping space. More explicitly since $\@X$ is cofibrant and spaces $E, F, B$ are fibrant, $\Rmap_*(-, -)$ is given by the simplicial function object $\Hom(-\x\Delta^{\Dot}, -)$ (see \cite[\S2.1, page 47]{MV99}). 
Let $\gamma\in \Rmap_*(\@X, (F, f))_0$ be a 0-simplex of the mapping space. Let  $\Rmap_*(\@X, (F, f))$ be pointed at $\gamma$ and $\Rmap_*(\@X, (E, e))$ be pointed at $p_*(\gamma)$. Consider the fiber sequence 
$$(\Rmap_*(\@X, (F, f)), \gamma)\xrightarrow{p_*} (\Rmap_*(\@X, (E, e)), p_*(\gamma))\xrightarrow{q_*} (\Rmap_*(\@X, (B, b)), *)$$ and the associated exact sequence 
$$\xymatrix{
\pi_1(\Rmap_*(\@X, (E, e)), p_*(\gamma))\ar[r] &\pi_1(\Rmap_*(\@X, (B, b)), *)\ar[r] & \pi_0(\Rmap_*(\@X, (F, f)), \gamma) \ar[d]^{p_*}\\
& & \pi_0(\Rmap_*(\@X, (E, e)), p_*(\gamma))
}$$
Since $(E, e)$ is an abelian group object with $m: (E, e)\x (E, e)\to (E, e)$, consider the map $$\xymatrix@C=3em{
(\Rmap_*(\@X, (E, e)), *) \ar[r]^-{Id \x p_*(\gamma)} \ar[rd]_-{m\circ (-, p_*(\gamma))}& (\Rmap_*(\@X, (E, e)\x (E, e)), (*, p_*(\gamma)))\ar[d]^{m} \\
& (\Rmap_*(\@X, (E, e)), p_*(\gamma)) 
}$$
Here the composite map of spaces $$m\circ (-, p_*(\gamma)): (\Rmap_*(\@X, (E, e)), *)\to (\Rmap_*(\@X, (E, e)), p_*(\gamma))$$ translating along base point $p_*(\gamma)$ induces the map 
$$T_{p_*(\gamma)}:=(m\circ (-, p_*(\gamma)))_*: \pi_1(\Rmap_*(\@X, (E, e)), *)\to \pi_1(\Rmap_*(\@X, (E, e)), p_*(\gamma)).$$
Now one can describe the inverse image of $p_*(\gamma)$.    
$$p_*^{-1}(p_*(\gamma))=\pi_1(\Rmap_*(\@X, (B, b)), *)/T_{p_*(\gamma)}(\pi_1(\Rmap_*(\@X, (E, e)), *)).$$
We will be using this crucially in the proof of ~\Cref{cancel2n-2}.

\subsection{Moore-Postnikov tower of a fiber sequence} The main tool we will be using in the paper is the existence of the Moore-Postnikov tower for a given fiber sequence in the homotopy category.

\begin{proposition}(\cite[Proposition 3.4]{PD22})\label{MP-tower} Let $k $ be a perfect field. Let $F\to E\xrightarrow{q} B$ be an $\A^1-$fiber sequence where $F, E, B$ are pointed $\A^1-$fibrant connected spaces. Let $G:=\piA_1(B)$. Then for each $n\geq 0$, there are fibrant connected spaces $E_n$ together with morphisms $q_n: E_n\to E_{n-1}$, $i_n: E\to E_n$ and, $p_n: E_n\to B$ forming (upto homotopy) a commutative diagram $$\xymatrix{
& \ar@{-->}[d] &\\
& E_{n+1} \ar[d]^{q_{n+1}} \ar[rd]^{p_{n+1}} & \\
E \ar[ru]^{i_{n+1}} \ar[r]^{i_n} \ar[rd]_{i_{n-1}} & E_n \ar[r]^{p_{n}} \ar[d]^{q_{n}} & B\\
& E_{n-1} \ar@{-->}[d]^{q_{n-1}} \ar[ru]_{p_{n-1}} & \\
& &
}$$ 
    such that 
    \begin{enumerate}
        \item $p_0: E_0\to B$ is an $\A^1-$weak equivalence.
        \item for each $n\geq 0$, $p_n\circ i_n=q$.
        \item The morphism $q_n$ is an $\A^1-$fibration with homotopy fiber as the Eilenberg-MacLane space $\K(\piA_n(F), n)$. Furthermore, for each $n\geq 2,$ there is (up to homotopy) a unique map $E_{n-1}\to \K(\piA_n(F), n+1)$ such that the diagram $$\xymatrix{
        E_n \ar[r] \ar[d]_{q_n} & {\rm BG} \ar[d] &\\
         E_{n-1} \ar[r]  & K^G(\piA_n(F), n+1) 
        }$$ is a homotopy pullback square. 
        \item The morphism $E\to \hlim\limits_{n} E_n$ induced by the maps $\{i_n: E\to E_n\}_{n\geq 0}$ is an $\A^1-$weak equivalence.
    \end{enumerate}
    
\end{proposition}
\begin{remark}
    In~\Cref{sec:cancel-and-split}, we apply Proposition~\ref{MP-tower} to the $\A^1-$fiber sequence $$F'_{n} \to \BSp_{2n}\to \BSp$$ for various natural numbers $n$ where $\BSp_{2n}$ denotes the classifying space associated to the group scheme $\Sp_{2n}$ and $\BSp=\dlim_{n} \BSp_{2n}$. Also, $F'_n$ denotes the homotopy fiber of the natural map $\BSp_{2n}\to \BSp.$ Note that $\piA_1(\BSp)=*$.  
\end{remark}

\subsection{Suslin matrices and motivic cohomotopy sets}\label{subsec:sus-matrices}
In this section we collect some important results about certain cohomotopy sets that will be used in~\Cref{sec:cancel-and-split}.

We recall for $r\geq 1$ $$Q_{2r-1}=\Spec \dfrac{k[x_1,\dots, x_{r}, y_1, \dots, y_r]}{\sum_{i=1}^{r} x_iy_i-1}.$$
For a $k-$algebra $R$, an $R$-valued point of the scheme $Q_{2r-1}$ is given by a tuple $(a, b)$ where $a=(a_1,\dots, a_r)\in R^{\oplus r}$ and $b=(b_1,\dots, b_r)\in R^{\oplus r}$ such that $\sum_{i=1}^r a_ib_i=1$ determined by $$x_i \mapsto a_i \ \text{and} \ y_i \mapsto b_i $$
for $i=1,\dots, r$.

Following~\cite[\S 5]{Suslin77} and~\cite[Chapter III, \S 7]{Lam06} Suslin constructed a map of $k-$schemes $$\alpha_{r}: Q_{2r-1} \to \SL_{2^{r-1}} \subset \GL_{2^{r-1}}  .$$ 
Up to elementary matrix operations on the rows and columns, Suslin reduced the matrix  $\alpha_{r}(a, b) \in \SL_{2^{r-1}}(R)$ for any $k-$algebra $R$ to $\beta_{r}(a, b)\in \SL_{r}(R)$, hence get a $k-$morphism of schemes
 $$\beta_{r}: Q_{2r-1} \to \SL_{r} \subset \GL_{r}  $$ \ie the following diagram 
$$\xymatrix{
& \SL_{r} \ar@{^(->}[d] \\
Q_{2r-1}\ar[r]^{\alpha_r} \ar[ru]^{\beta_r}& \SL_{2^{r-1}}
}$$
commutes on $R-$points up to elementary matrix operations on the set $\SL_{2^{r-1}}(R)$. The projection onto the first $r-$coordinates induces the map $Q_{2r-1}\to \#A^{r}\setminus\{0\}$ which is Zariski locally trivial fiber bundle with fiber $\#A^{r-1}$, hence is an $\A^1-$weak equivalence as noted in~\Cref{quadric-hom-type}. Consider the composite map in $\@H^{\#A^1}(k)$
$$\psi_{r}: \#A^{r}\setminus\{0\} \simeq_{\#A^1} Q_{2r-1}\xrightarrow{\beta_r}  \SL_{r}\xrightarrow{pr_1} \#A^{r}\setminus\{0\}  $$
where $pr_1: \SL_r\to \#A^{r}\setminus\{0\} $ is the projection sending an $r\x r-$square matrix to its first row.
More explicitly, the map $\psi_{r}: \#A^{r}\setminus\{0\}\to \#A^{r}\setminus\{0\}$ is given by $$(x_1, \dots, x_{r})\to (x_1^{(r-1)!}, \dots, x_{r}).$$
This induces for any $k-$smooth scheme $X$ the map on cohomotopy sets
$${\psi_{r}}_*: [X, \#A^{r}\setminus\{0\}]_{\#A^1}\to [X, \#A^{r}\setminus\{0\}]_{\#A^1}$$ sending $$[\alpha]\mapsto [\psi_{r}\circ \alpha].$$

In the following proposition, we record the assumptions under which the map ${\psi_r}_*$ is explicit. \begin{proposition}\label{prop:suslin-mult} Let $X$ be a smooth affine $k-$scheme of $\dim X \geq 2$.
If either of the following conditions hold \begin{enumerate}
    \item $2\leq \dim X\leq r-1$ 
    or \item $2\leq \dim X=r$, the field $k$ is algebraically closed and $r!\in k^\x$     
\end{enumerate} 
then under the map $${\psi_{r}}_*: [X, \#A^{r}\setminus\{0\}]_{\#A^1}\to [X, \#A^{r}\setminus\{0\}]_{\#A^1}$$ 
$${\psi_{r}}_*([a_1, \dots, a_{r}])=\dfrac{(r-1)!}{2}\cdot h\cdot ([a_1, \dots, a_{r}])=(r-1)!\cdot ([a_1, \dots, a_{r}]).$$
For the last equality, note that since the field $k$ is algebraically closed, $h=\langle1\rangle+\langle-1\rangle=2.$
\end{proposition}
 \begin{proof}
   For $2\leq \dim X \leq r-1 $, the reader can refer to the proof in~\cite[Page 14838 and Remark 6.12]{PD22}. 
   For the case $\dim X=r$, the proof follows along the same lines as in~\cite[Page 14838 and Remark 6.12]{PD22}, except that one uses~\cite[Theorem 3]{Fas24}, hence the assumptions on the field. 
For the convenience of the reader we provide the argument. 

Consider the map $w: \#A^{r}\setminus\{0\}\to \K(\KMW_r, r-1)$ (to the first non-trivial $(r-1)-$th term $\K(\KMW_r, r-1)$ of the Moore-Postnikov tower for the map $\#A^{r}\setminus\{0\}\to *.$) 
By~\cite[Lemma 4.5]{AF14-JTop} the isomorphism $$[\#A^{r}\setminus\{0\}, \K(\KMW_r, r-1)]_{\#A^1}\simeq \coh^{r-1}(\#A^{r}\setminus\{0\}, \KMW_r)$$ the map $w$ corresponds to cohomology class $[w]$ and  
the cohomology group   $\coh^{r-1}(\#A^{r}\setminus\{0\}, \KMW_r)$ is rank 1 free $\KMW_0(k)-$module generated by $[w]$.  

As noted in ~\cite[Remark 6.12]{PD22}, 
\begin{equation}\label{eq:univ-w}
    \psi_{r}^*([w])=\dfrac{(r-1)!}{2}\cdot h\cdot [w] \end{equation} 
    in $\coh^{r-1}(\#A^{r}\setminus\{0\}, \KMW_r).$    
Denote by $[a]: X\to \#A^{r}\setminus\{0\}$ the map associated to the tuple $(a_1, \dots, a_{r})$. As a consequence of~\eqref{eq:univ-w}, we have  
\begin{equation}\label{eq:suslinmap}
    w_*({\psi_r}_*([a]))=[w\circ \psi_r\circ a]=\psi_{r}^*([w])([a])=\bigg(\dfrac{(r-1)!}{2}\cdot h\cdot [w]\bigg)([a])= w_*\bigg(\dfrac{(r-1)!}{2}\cdot h\cdot[a]\bigg).
    \end{equation}
   Note that the induced map $$w_*: [X, \#A^{r}\setminus\{0\}]_{\#A^1}\to [X, \K(\KMW_r, r-1)]_{\#A^1}\simeq \coh^{r-1}(X, \KMW_r)$$ is isomorphism under our assumptions. To see this, recall that $$w:  \#A^{r}\setminus\{0\}\to \K(\KMW_r, r-1)$$ is the map to the first non-trivial $(r-1)-$th term $\K(\KMW_r, r-1)$ of the Moore-Postnikov tower for the map $\#A^{r}\setminus\{0\}\to *$  and we have the following fiber sequence
   \begin{equation}\label{seq:fiber2.4}
       \K(\piA_r(\#A^{r}\setminus\{0\}), r) \to (\#A^{r}\setminus\{0\})[r]\xrightarrow{q_{r}} \K(\KMW_r, r-1)
        \end{equation}
with commutative diagram 
$$\xymatrix{
& \#A^{r}\setminus\{0\}\ar[d]^{i_r} \ar[rd]^w &\\
\K(\piA_r(\#A^{r}\setminus\{0\}), r) \ar[r] & (\#A^{r}\setminus\{0\})[r]\ar[r]^{q_{r}} &\K(\KMW_r, r-1)
}$$
By~\cite[Appendix B]{Mor12}, it follows that since $\dim X=r$, the map
$${i_r}_*: [X, \#A^{r}\setminus\{0\}]_{\#A^1}\to [X, (\#A^{r}\setminus\{0\})[r]]_{\#A^1}$$
is an isomorphism. Consider the long exact sequence associated to the fiber sequence~\eqref{seq:fiber2.4}
$$[X, \K(\piA_r(\#A^{r}\setminus\{0\}), r)]_{\#A^1}\to [X, \#A^{r}\setminus\{0\}]_{\#A^1}\xrightarrow{w_*} [X, \K(\KMW_r, r-1)]_{\#A^1}. $$
  Note that  $[X, \K(\piA_r(\#A^{r}\setminus\{0\}), r)]_{\#A^1}\simeq \coh^r(X, \piA_r(\#A^{r}\setminus\{0\}))$ and by~\cite[Theorem 3]{Fas24}, since under our assumptions, the field $k$ is algebraically closed and $r!\in k^\x$, we have $\coh^r(X, \piA_r(\#A^{r}\setminus\{0\}))=0.$ Thus we deduce that the map $w_*: [X, \#A^{r}\setminus\{0\}]_{\#A^1}\to [X, \K(\KMW_r, r-1)]_{\#A^1} $
  is injective and by \eqref{eq:suslinmap}, we have 
  $${\psi_{r}}_*([a])=\dfrac{(r-1)!}{2}\cdot h\cdot [a]=(r-1)!\cdot ([a_1, \dots, a_{r}])$$
  in the cohomotopy set $[X, \#A^{r}\setminus\{0\}]_{\#A^1}$. Note that the last equality holds as $h=2$ under the assumption that $k=\-k.$ 
 \end{proof}

\section{Some vanishing results}\label{vanishing-sec}
In this section, we organize some cohomology vanishing results for certain homotopy sheaves of a particular class of spheres in the motivic homotopy category.

For a field $k$, let us recall the affine quadrics $$Q_{2n-1}=\Spec \dfrac{k[x_1, y_1, \dots, x_{n}, y_n]}{\sum_{i=1}^n x_iy_i-1} $$
and $$Q_{2n}=\Spec \dfrac{k[x_1, y_1, \dots, x_{n}, y_n, z]}{\sum_{i=1}^n x_iy_i-z(1-z)}.$$
For a $k-$algebra $R$, an $R$-valued point of the scheme $Q_{2n-1}$ is given by a tuple $(a, b)$ where $a=(a_1,\dots, a_n)\in R^{\oplus n}$ and $b=(b_1,\dots, b_n)\in R^{\oplus n}$ such that $\sum_{i=1}^n a_ib_i=1$ determined by $$x_i \mapsto a_i \ \text{and} \ y_i \mapsto b_i $$
for $i=1,\dots, n$.
We consider  $Q_{2n}$ and $Q_{2n-1}$ as pointed spaces in the motivic homotopy category and recall the following proposition which will be used repeatedly. 
\begin{proposition}(\cite[Proposition 1.1.5]{AF22}) \label{quadric-hom-type} Let $k$ be a field. Then the following holds.
\begin{enumerate}
\item[$(1)$] The quadric $Q_m$ is a smooth affine $k-$scheme of dimension $m$. 
    \item[$(2)$] 
There are $\#A^1-$weak equivalences
\begin{align*}
    Q_m &\sim_{\#A^1}\begin{cases}
        S^{n-1}\wedge \#G_m^{\wedge n} \sim_{\#A^1} \#A^n\setminus\{0\} &  \text{if} \ m=2n-1 \ \text{and}\\
        S^{n}\wedge \#G_m^{\wedge n} \sim_{\#A^1} (\#P^1)^{\wedge n}& \text{if} \ m=2n
    \end{cases}
\end{align*}
    of spaces over $\Spec k$.
\item[$(3)$] The quadric $Q_m$ is at least $\#A^1-(\lfloor{\frac{m}{2}\rfloor}-1)-$connected. Moreover, the first non-trivial $\#A^1-$homotopy sheaf of $Q_m$ appears in degree $\lfloor{\frac{m}{2}\rfloor}$ and if $m=2n$ or $m=2n-1$, it is $\KMW_n.$ 
    \item[$(4)$] If $n>0$, for a strictly $\#A^1-$invariant sheaf ${\bf M}$ of abelian groups  \begin{align*}
\coh^{i}(Q_{2n}, {\bf M})=&   \begin{cases}
        {\bf M}(k) \ & \text{if} \ i=0\\
       {\bf M}_{-n}(k) \ &\text{if} \ i=n \\
       0 \ & \text{otherwise}
    \end{cases}
    \end{align*}
    and for $n \geq 2$,
  \begin{align*}
\coh^{i}(Q_{2n-1}, {\bf M})=&   \begin{cases}
        {\bf M}(k) \ & \text{if} \ i=0\\
       {\bf M}_{-n}(k) \ &\text{if} \ i=n-1 \\
       0 \ & \text{otherwise.}
    \end{cases}
\end{align*}
\end{enumerate}
\end{proposition}

In this section, we mainly consider the following question.

\begin{question}(Vanishing of $\coh^d(X, \piA_d(Q_{2d-2}))$)\label{Q-Q} Let $X$ be an affine scheme of dimension $d$ over a field $k$. 
Under what conditions on $k$, $X$, $d$, is
$$\coh^d(X, \piA_d(Q_{2d-2}))=0?$$
\end{question}
The need to establish this will be clearer in  section~\ref{Euler-Chow}. 
In the vein of answering the above question, we prove the following proposition. 
\begin{proposition}\label{adjoint}
Let $k$ be a field that is not formally real. Then the map of Nisnevich sheaves 
$$\piA_{d-1}(\A^{d-1}\setminus \{0\})\to \piA_{d}(Q_{2d-2})$$ 
\begin{enumerate}
    \item[$(1)$] is an epimorphism for $d\geq 4$ and
    \item[$(2)$]  an isomorphism for $d\geq 5.$
\end{enumerate}

\end{proposition}

    \begin{proof}
        
 By \cite[Proposition 5.3.10 (1)]{AFH22} for $d\geq 3$ we have a fiber sequence  $$Q_{2d-3}\to \Omega Q_{2d-2} \to \Omega Q_{4d-5}$$ 
hence the exact sequence of sheaves of homotopy groups 
$$\xymatrix{
\piA_{d+1}(Q_{4d-5})\ar[r] & \piA_{d-1}(Q_{2d-3})\ar[r] & \piA_d(Q_{2d-2}) \ar[r] & \piA_d(Q_{4d-5})\ar[r]\ar[d]^{\simeq} & \piA_{d-2}(Q_{2d-3})\\ 
& & & \piA_d(\A^{2d-2}\setminus\{0\})  & 
}$$
The $\A^1-$weak equivalence $Q_{4d-5}\simeq S^{2d-3}\wedge \G_m^{\wedge 2d-2}\simeq \A^{2d-2}\setminus \{0\}$ (see~\Cref{quadric-hom-type}) induces the isomorphism 
$$\piA_d(Q_{4d-5})\simeq \piA_d(\A^{2d-2}\setminus\{0\})\simeq 0 \ \text{for} \ d\geq 4$$ and
$$\piA_{d+1}(Q_{4d-5})\simeq \piA_{d+1}(\A^{2d-2}\setminus\{0\})\simeq 0 \ \text{for} \ d\geq 5.$$
Hence we have an epimorphism $$\piA_{d-1}(Q_{2d-3})\to \piA_d(Q_{2d-2}) $$ of Nisnevich sheaves if $d\geq 4$ and isomorphism for $d\geq 5$.
By~\Cref{quadric-hom-type}, we have an $\#A^1-$weak equivalence $Q_{2d-3} \sim_{\A^1} \A^{d-1}\setminus\{0\}$ we have for $d\geq 4$
the epimorphism of Nisnevich sheaves 
$$\piA_{d-1}(\A^{d-1}\setminus\{0\})\to \piA_{d}(Q_{2d-2})$$ and it is an isomorphism for $d\geq 5.$

    \end{proof}

We recall the following result from  \cite{ABH23}.
\begin{proposition}
    \label{Vanishing:ABH}
(Asok--Bachmann--Hopkins~\cite{ABH23}) 
    Let $X$ be a smooth affine scheme of dimension $d$ ($d\geq 3$) over an algebraically closed field $k$ with $char(k)=0$. Then $$\coh^{d}(X, \piA_{d-1}(\A^{d-1}\setminus\{0\}))=0.$$
\end{proposition}
\begin{proof}
    This is implicit in the proof of \cite[Theorem 7.1.1]{ABH23}.
\end{proof}

\begin{lemma}\label{lemma:torsion} Let $X$ be a smooth affine scheme of dimension $d$ ($d\geq 3$) over an algebraically closed field $k$. 

Consider $\GL_{d-1}\to \GL$ and the induced map $\piA_{d-1}(\GL_{d-1})  \to \piA_{d-1}(\GL)$ of $\#A^1-$homotopy sheaves. Then the kernel of the map of cohomology groups
 $$\coh^d(X, \piA_{d-1}(\GL_{d-1}))  \to \coh^d(X, \piA_{d-1}(\GL))$$ is $d!\cdot (d-1)!-$torsion. 

 Moreover, if we further assume that $d! \in k^{\x}$ then the kernel is $ (d-1)!-$torsion.
\end{lemma}
\begin{proof}
Consider the composite map $$ \coh^d(X, \piA_{d-1}(\GL_{d-1}))\xrightarrow{f} \coh^d(X, \piA_{d-1}(\GL_{d}))\xrightarrow{g} \coh^d(X, \piA_{d-1}(\GL_{d+1}))\xrightarrow{h} \coh^d(X, \piA_{d-1}(\GL)).$$ First we note that the map $h$ is an isomorphism, as follows from the isomorphism of the sheaves $\piA_{d-1}(\GL_{d+1})\simeq \piA_{d-1}(\GL)$.   
 To see that the map $g$ is an isomorphism, consider the fiber sequence  $$\GL_d\to \GL_{d+1}\to \A^{d+1}\setminus\{0\} $$ hence the long exact sequence $$ \piA_d(\GL_{d+1})\to \piA_d(\A^{d+1}\setminus\{0\})\to \piA_{d-1}(\GL_{d}) \to \piA_{d-1}(\GL_{d+1}) \to \piA_{d-1}(\A^{d+1}\setminus\{0\})=0.$$ 
 Breaking the long exact sequence as 
 $$\piA_d(\GL_{d+1})\to \piA_d(\A^{d+1}\setminus\{0\})\to A \to 0$$ and
 $$ 0\to A\to  \piA_{d-1}(\GL_{d}) \to \piA_{d-1}(\GL_{d+1})\to 0$$ and considering the associated long exact sequence of cohomology groups, we can deduce that the sequences 
 \begin{equation}\label{eq:coker1}
     \coh^d(X, \piA_{d}(\GL_{d+1}))\to \coh^d(X, \piA_d(\A^{d+1}\setminus\{0\}))\to \coh^d(X, A)\to 0
     \end{equation} and 
     \begin{equation}\label{eq:coker2}
          \coh^d(X, A)\to \coh^d(X, \piA_{d-1}(\GL_{d}))\xrightarrow{g} \coh^d(X, \piA_{d-1}(\GL_{d+1}))\to 0
          \end{equation} are exact.
 By~\cite[Lemma 2.4.1]{Fas24} and \eqref{eq:coker1} $\coh^d(X, A)$ is $d!-$torsion, hence by \eqref{eq:coker2} $\ker g$ is $d!-$torsion. (If we furthermore assume that $d! \in k^{\x}$, hence in particular the field $k$ is of characteristic $\neq 2$ (as $d\geq 3$),  by~\cite[Theorem 1.4, Theorem 2.2]{Fas15}, the cohomology group $$\coh^d(X,  \piA_d(\A^{d+1}\setminus\{0\}))\simeq \coh^d(X, \KMW_{d+1})\simeq \coh^d(X, \KM_{d+1})$$ is uniquely divisible prime to the characteristic of $k$. Hence by the surjection in \eqref{eq:coker1}, we can deduce that $\coh^d(X, A)$ is divisible prime to the characteristic of $k$, in particular it is $d!-$divisible. Thus, combining with $\coh^d(X, A)$ is $d!-$torsion, we deduce that $\coh^d(X, A)=0$ and hence $\Ker(g)=0$.)

 To see that the kernel $\Ker(f)$ of the homomorphism $f$
 is $(d-1)!-$torsion, consider the fiber sequence 
  $$\GL_{d-1}\to \GL_{d}\to \A^{d}\setminus\{0\} $$ hence the long exact sequence $$ \piA_d(\GL_{d})\to \piA_d(\A^{d}\setminus\{0\})\to \piA_{d-1}(\GL_{d-1}) \to \piA_{d-1}(\GL_{d}).$$ 
  Similarly as above, consider the long exact sequence of cohomology groups and get an exact sequence 
  $$\coh^d(X, \piA_d(\GL_{d}))\to \coh^d(X, \piA_d(\A^{d}\setminus\{0\}))\to \coh^d(X, \piA_{d-1}(\GL_{d-1}))\xrightarrow{f} \coh^d(X, \piA_{d-1}(\GL_{d}))$$
  By \cite[Lemma 2.4.1]{Fas24}, the cokernel of the map 
  $$\coh^d(X, \piA_d(\GL_{d}))\to \coh^d(X, \piA_d(\A^{d}\setminus\{0\})) $$ is $(d-1)!-$torsion. Hence $\Ker(f)$ is $(d-1)!-$torsion. 

  Combining the above we show that $\Ker(h\circ g\circ f)$ is $d!\cdot (d-1)!-$torsion. 
  To see this, consider the exact sequence $$0\to  \Ker(f)\to  \Ker(g\circ f) \to \Ker(g) \to \Coker(f)\to \Coker (g\circ f) \to \Coker(g) \to 0.$$ Using that $\Ker (f)$ is $(d-1)!-$torsion and $\Ker (g)$ is $d!-$torsion, we deduce that $\Ker (h\circ g\circ f)=\Ker (g\circ f)$ is $d!(d-1)!-$torsion. (As noted earlier, under the additional assumption that $d! \in k^{\x}$, we can deduce that $\Ker(h\circ g\circ f)=\Ker (g\circ f)=\Ker (f)$ is $(d-1)!-$torsion.)
\end{proof}
\begin{proposition} \label{torsion}
    Let $X$ be a smooth affine scheme of dimension $d$ ($d\geq 5$) over an algebraically closed field $k$ of characteristic $\neq 2$. Then $\coh^{d}(X, \piA_{d-1}(\A^{d-1}\setminus\{0\}))$ is $d!\cdot (d-1)!\cdot (d-2)!-$torsion.

    Moreover, if we further assume that $d! \in k^{\x}$ then the cohomology group is $ (d-1)!\cdot(d-2)!-$torsion.
\end{proposition}
\begin{proof}
The proof follows along the lines as in ~\cite[Proposition 4.0.9]{Fas24}. We provide the details for the convenience of the reader. 

Let $O$ be the (infinite) orthogonal group. Consider the Suslin map \cite[proof of Lemma 3.10]{AF14-JTop} $$u_{d-1}: \A^{d-1}\setminus\{0\} \to \GL_{d-1}\to  \GL $$ from \cite[section 2.4]{Fas24} which from \cite[\S 3]{AF17} factors through
$$ \A^{d-1}\setminus\{0\} \to \Omega_{\#P^1}^{-d+1}O \to  \GL$$
where $\Omega_{\#P^1}^{-d+1}O $ is a space representing Hermitian K-theory defined in~\cite[Definition 2.2.3]{AF17} and $\Omega_{\#P^1}^{-d+1}O \to \GL $ is the map induced by the forgetful map from Hermitian K-theory to K-theory. 
Now by ~\cite[Theorem 5 and Theorem 6]{ST15} and \cite[\S 4.4]{AF17} 
 $\piA_{d-1}(\Omega_{\#P^1}^{-d+1}O)\simeq \&{GW}_d^{d-1}$ and by \cite[Theorem 3.7.1]{AF15}  $\coh^d(X, \&{GW}_d^{d-1})=0$ (since $\CH^d(X)/2=0$ as base field $k$ is algebraically closed of characteristic $\neq 2$). In particular, this implies that the induced map on cohomology groups
 $$\coh^d(X, \piA_{d-1}( \A^{d-1}\setminus\{0\} )) \xrightarrow{\phi}  \coh^d(X, \piA_{d-1}(\GL_{d-1}))  \to \coh^d(X, \piA_{d-1}(\GL))$$ is 0.
  By~\Cref{lemma:torsion} the kernel of the 2nd map
 $$ \coh^d(X, \piA_{d-1}(\GL_{d-1}))  \to \coh^d(X, \piA_{d-1}(\GL)$$ is $ d!(d-1)!-$torsion.  Hence, the image of the map $\phi$ is $d!\cdot (d-1)!-$torsion. (Again by~\Cref{lemma:torsion}, under the additional assumption that $d! \in k^{\x}$, the image of the map $\phi$ is $(d-1)! -$torsion.)

  Further consider the composite of the maps from \cite[section 2.4]{Fas24} $$\A^{d-1}\setminus\{0\}\xrightarrow{u_{d-1}'} \GL_{d-1}\xrightarrow{r_{d-1}} \A^{d-1}\setminus\{0\}$$ $$(x_1, \dots, x_{d-1})\mapsto (x_1^{(d-2)!}, \dots, x_{d-1})$$ 
  is given by $\mu_{(d-2)!}.$ Since $\sqrt{-1} \in k $ and $d\geq 5$, by~\cite[Lemma 2.3.3]{Fas24}, the induced composite map 
 $$\coh^d(X, \piA_{d-1}(\A^{d-1}\setminus\{0\}))\xrightarrow{\phi} \coh^d(X, \piA_{d-1}(\GL_{d-1}))\to \coh^d(X, \piA_{d-1}(\A^{d-1}\setminus\{0\}))$$
 is given by multiplication by $(d-2)!$. As noted above the image of the first map $\phi$ is $d!\cdot (d-1)!-$torsion. 

 Thus combining the above observations, we get that the cohomology group $\coh^d(X, \piA_{d-1}(\A^{d-1}\setminus\{0\}))$ is $d!\cdot(d-1)!\cdot(d-2)!-$torsion. (Under the assumption $d! \in k^{\x}$, as noted above the image of the map $\phi$ is $(d-1)!-$torsion. Hence the group $\coh^d(X, \piA_{d-1}(\A^{d-1}\setminus\{0\}))$ is $(d-1)!\cdot(d-2)!-$torsion.)
\end{proof}

\begin{remark}
Fasel in \cite[Theorem 3]{Fas24} proved that under the assumptions of~\Cref{torsion}, if $d!\in k^{\x}$, then $\coh^d(X, \piA_d(\A^{d}\setminus\{0\}))=0$.
 \end{remark}
As a consequence, we have the following result.
\begin{corollary}\label{cor:vanishing}
     Let $X$ be a smooth affine scheme of dimension $d$ over an algebraically closed field $k$ of characteristic $\neq 2$. Then
     \begin{enumerate}
         \item[$(1)$] if $char(k)=0$ and $d\geq 4$, then  $\coh^{d}(X, \piA_{d}(Q_{2d-2}))=0$ and
         \item[$(2)$] if $d\geq 5$, then $\coh^{d}(X, \piA_{d}(Q_{2d-2}))$ is $d!\cdot (d-1)!\cdot(d-2)!-$torsion.
         \item[$(3)$] if $d\geq 5$ and $d!\in k^{\x}$, then $\coh^{d}(X, \piA_{d}(Q_{2d-2}))$ is $(d-1)!\cdot(d-2)!-$torsion.
     \end{enumerate}
\end{corollary}
\begin{proof}
  (1) follows from Proposition \ref{Vanishing:ABH} and  Proposition~\ref{adjoint}.
        The statements (2) and (3) follow from Proposition~\ref{torsion} and  Proposition~\ref{adjoint}.
    \end{proof}
\begin{remark} 
In general, though we suspect it to be so, it seems difficult to determine whether the cohomology groups $\coh^d(X, \piA_{d-1}(\A^{d-1}\setminus\{0\}))$ or $\coh^d(X, \piA_d(Q_{2d-2}))$ are trivial or not.
\end{remark}

We note that the vanishing of the cohomology groups $$\coh^d(X, \piA_{d-1}(\A^{d-1}\setminus\{0\})) \ \text{or} \ \coh^d(X, \piA_d(Q_{2d-2}))$$ follows from the following deep conjecture. 
\begin{conjecture}\label{AF-conj}(Asok--Fasel~\cite[Conjecture 7]{AF13})
 For $d\geq 5,$ the following sequence
$$(*) \ \ \ \KM_{d+1}/24\to \piA_{d-1}(\A^{d-1}\setminus\{0\})\to \&{GW}^{d-1}_d$$
is exact and after $(d-4)-$fold contractions it is surjective on the right.
\end{conjecture}

\begin{proposition}\label{Q-vanish-mod-conj} 
Assume Conjecture~\ref{AF-conj} holds. Let $k$ be an algebraically closed field of characteristic $\neq 2$. Let $X$ be a smooth affine variety of dimension $d\geq 5$ over $k$. Then 
$$\coh^d(X, \piA_d(Q_{2d-2}))=0.$$
\end{proposition}
\begin{proof}
   By~\Cref{adjoint}, we have the epimorphism of sheaves $$\piA_{d-1}(\A^{d-1}\setminus\{0\})\to \piA_{d}(Q_{2d-2}).$$ Hence it suffices to show that $\coh^d(X, \piA_{d-1}(\A^{d-1}\setminus\{0\}))=0$. 
   
In view of Conjecture~\ref{AF-conj}, for $d\geq 5,$ the following sequence 
$$ \KM_{d+1}/24\to \piA_{d-1}(\A^{d-1}\setminus\{0\})\to \&{GW}^{d-1}_d$$
is exact and surjective on the right after $(d-4)-$fold
contractions. Hence it is enough to prove the vanishing of $\coh^d(X,\&{GW}^{d-1}_d)$ and $\coh^d(X, \KM_{d+1}/24)$.

The vanishing of $\coh^d(X,\&{GW}^{d-1}_d)$ follows from~\cite[Proposition 3.6.4]{AF15} under the assumptions of the base field $k$. For this we will need to show that $\CH^d(X)/2=0$. Since the cohomological dimension of $X$ is at most $d$, we can observe that $$\CH^d(X)/2\simeq \coh^d(X, \KM_d/2).$$ 
The vanishing of $\coh^d(X,\KM_{d}/2)$ follows from the Gersten complex for the sheaf $\KM_{d}/2$. The group $\coh^d(X,\KM_{d}/2)$ is a quotient of $\bigoplus_{x\in X^{(d)}}k(x)^*/2$ and since $k(x)=k$ is an algebraically closed field, $k(x)^*/2=0$ for each $x\in X$ of codimension $d$.

It remains to show the vanishing of $\coh^d(X,\KM_{d+1}/24)$. This follows from the Gersten complex for the sheaf $\KM_{d+1}/24$. The cohomology group $\coh^d(X,\KM_{d+1}/24)$ is a quotient of $\bigoplus_{x\in X^{(d)}}k(x)^*/24=0 $ since $k(x)=k$ is an algebraically closed field.
\end{proof}

For the rest of this section we answer the Question~\ref{Q-Q} for $d=2, 3$ and 4.
\begin{proposition}\label{dim2}
Let $X$ be a smooth affine scheme of dimension $d=2$ over an algebraically closed field $k$ of characteristic $\neq 2$. Then $$\coh^2(X, \piA_2(Q_{2}))=0.$$
\end{proposition}
\begin{proof}

From \Cref{quadric-hom-type}, we have $Q_2 \sim_{\A^1}S^1\wedge \#G_m \sim_{\A^1} \mathbb{P}^1.$ 
We have an $\A^1-$fiber sequence $$\#A^2 \setminus\{0\}\to \#P^1\to B\#G_m$$
 which induces the long exact sequence of sheaves
 $$\cdots \to \piA_3(B\#G_m)\to \piA_2(\A^2\setminus\{0\})\to  \piA_2(\#P^1)\to \piA_2(B\#G_m)\to \cdots $$
 One observes that $\piA_3(B\#G_m)=0$ and $ \piA_2(B\#G_m)=0,$ hence we have the isomorphisms $$\piA_2(\A^2\setminus\{0\})\simeq  \piA_2(\#P^1)\simeq  \piA_2(Q_2)$$ of sheaves. Furthermore, since $\A^2\setminus\{0\}\sim_{\#A^1} \SL_2$, we have $$\piA_2(Q_2)\simeq \piA_2(\A^2\setminus\{0\})\simeq \piA_2(\SL_2).$$ Hence it is enough to show that $\coh^2(X,\piA_2(\SL_2))=0$.
From \cite[Theorem 3.3]{AF14}, we have the following short exact sequence of the form 
$$0\to \&T_4'\to \piA_2(\SL_2)\to \&{GW}^2_3\to 0.$$
Hence we get the long exact sequence of cohomology groups
$$\coh^2(X, \&T'_4)\to \coh^2(X, \piA_2(\SL_2))\to \coh^2(X,\&{GW}^2_3)\to 0.$$
This reduces to show that under the assumptions of the proposition $$\coh^2(X, \&T'_4)=0 \ \ \text{and} \ \ \coh^2(X,\&{GW}^2_3)=0.$$
Now we have the following sequence by~\cite[Lemma 3.1]{AF14}
$$\&I^5\to \&T'_4\to \&S_4'\to 0$$ which is exact (not necessarily on the left) and the sheaf $\&S'_4$ is a quotient of $\KM_4/12$ where $\KM_4$ is the Milnor K-theory sheaf. By~\cite[Lemma 3.0.1]{Fas24} the sheaf $\&I^5=0$ once restricted to the small Nisnevich site associated to $X$ of $\dim X= 2$.  Hence $\&T'_4\to \&S'_4$ is an isomorphism of sheaves on $X$. Thus $\coh^2(X, \&T'_4)\simeq \coh^2(X, \&S'_4)$.  Recalling that $\&S'_4$ is a quotient of $\KM_4/12$, by the long exact sequence of the cohomologies and also that the Nisnevich cohomological dimension of $X$ is at most 2, we get  $$\coh^2(X, \KM_4/12)\twoheadrightarrow \coh^2(X, \&S_4')$$ is surjective.

We calculate 
$\coh^2(X, \KM_4/12)$ using the Gersten complex resolution for $\KM_4/12.$ Thus, it will a quotient of $$\bigoplus_{x\in X^{(2)}}\KM_2(k(x))/12.$$ We observe that for $x\in X^{(2)}$,  $\KM_2(k(x))/12=0$ follows from \cite[Chapter VI, Corollary 1.3.1, Theorem 1.6]{Weib13}  since $k(x)=k$ is an algebraically closed field.
Thus, we get $\coh^2(X, \KM_4/12)=0$, and hence $\coh^2(X, \&S_4')=0.$ 

Now it remains to show $\coh^2(X, \&{GW}^2_3)=0.$
We have the following short exact sequence from \cite[section 3.6, page 1044]{AF15}
$$0\to \&{A} \to \&{GW}_3^2 \to \&{B} \to 0.$$
We refer the reader to \cite[section 3.6]{AF15} and also \cite[(4.4), page 2579]{AF14}  for the definition and properties of $\&A$ and $\&B$, which will be used crucially in what follows.

By the above exact sequence, to prove $\coh^2(X, \&{GW}_3^2)=0$, it is enough to prove that $$\coh^2(X, \&{A})=0 \ \ \text{and} \ \  \coh^2(X, \&{B})=0.$$
By \cite[Lemma 4.15]{AF14}, we have $\coh^2(X, \&B)\cong \coh^2(X, \&K_2^Q/2)$. 
We have an exact sequence $$0\to 2\cdot \&K_2^Q\to \&K_2^Q\to \&K_2^Q/2\to 0$$
Hence $$ \coh^2(X, \&K_2^Q)\xrightarrow{2\cdot} \coh^2(X, \&K_2^Q)\twoheadrightarrow  \coh^2(X, \&K_2^Q/2)$$ 
is exact. We observe that by Bloch's formula, $$\coh^2(X, \&K_2^Q)=\CH^2(X)\simeq F^2K_0(X).$$ 
Since $F^2K_0(X)$ is 2-divisible for a smooth affine variety $X$ over an algebraically closed field, we conclude that $\coh^2(X, \&K_2^Q/2)=0$.
Hence $\coh^2(X, \&B)=0$.

Thus, it remains to prove $\coh^2(X, \&A)=0$. Note that by~\cite[Lemma 4.11]{AF14} $$(\&A)_{-2}\simeq\&K_1^{Q}/2.$$
From the Gersten resolution for the sheaf $\&A$, we have the following surjection

$$\bigoplus_{x\in X^{(2)}}\&K_1^{Q}/2(k(x))\simeq \bigoplus_{x\in X^{(2)}}(\&A)_{-2}(k(x)) \twoheadrightarrow \coh^2(X,\&A)$$ 
Since the field $k(x)=k$ is an algebraically closed field, the left- hand side group vanishes.
Hence $\coh^2(X,\&A)=0$. Combining all the vanishing results above we get the proposition.
\end{proof}

We note the following vanishing result.
\begin{proposition}\label{van1}
Let $X$ be a smooth affine scheme of dimension 3 over an algebraically closed field $k$ of characteristic $\neq 2$.
Then $\coh^3(X, \piA_2(Q_3))\simeq \coh^3(X, \piA_2(\A^2\setminus\{0\}))=0$.
    
\end{proposition}
\begin{proof}
    Using that $\#A^2-0\sim_{\#A^1}\SL_2=\Sp_2$ we have the following short exact sequence from  \cite[Theorem 3.3]{AF14}
$$0\to \&T_4'\to \piA_2(\A^2\setminus\{0\})\to \&{GW}^2_3\to 0$$
hence we get the long exact sequence of cohomology groups
$$\coh^3(X, \&T'_4)\to \coh^3(X, \piA_2(\A^2\setminus\{0\}))\to \coh^3(X,\&{GW}^2_3)\to \coh^4(X,\&T_4')=0 \ \text{~~(because of dimension)} $$
This reduces to show that under the same assumptions,  we have $$\coh^3(X, \&T'_4)=0 \ \ \text{and} \ \ \coh^3(X,\&{GW}^2_3)=0$$ which is shown in the last paragraph of the proof of~\cite[Theorem 6.6]{AF14}.
Hence the proposition.
\end{proof}

For $d=3$ in the case of the field $k=\-{\#F}_p$ we have the following vanishing result.
\begin{proposition}\label{vanishing-dim3}
    Let $X$ be a smooth affine scheme of dimension $3$ over an algebraically closed field $k$.
 If $k=\overline{\mathbb{F}}_p$, for a prime $p\neq 2$, then $\coh^{3}(X, \piA_3(Q_{4}))=0$.
\end{proposition}
\begin{proof}
We recall from \cite[Example 4.3.1]{ADF17} that up to change of co-ordinates $$Q_{7}=\Spec \bigg(\dfrac{k[x_{11}, x_{12}, x_{21}, x_{22}, y_{11}, y_{22}, y_{12}, y_{21}]}{ x_{11}x_{22}-x_{12}x_{21}-y_{11}y_{22}+y_{12}y_{21}-1}\bigg) $$
and $$Q_{4}=\Spec \bigg(\dfrac{k[z_1, z_2, z_3, z_4, z_5]}{z_1z_4-z_2z_3-z_5(1+z_5)}\bigg).$$
The motivic Hopf map $\nu: Q_7\to Q_4$ is given by  $$(M_1, M_2) \mapsto (M_1M_2, \det M_2).$$  
where $M_1, M_2$ are $2\x2$ matrices $M_1=\begin{pmatrix}
        x_{11} & x_{12}\\
    x_{21} & x_{22}
\end{pmatrix}$, $M_2=\begin{pmatrix}
        y_{11} & y_{12}\\
    y_{21} & y_{22}
\end{pmatrix}$ such that $\det M_1-\det M_2=1$ and setting $\begin{pmatrix}
        z_{1} & z_2\\
    z_{3} & z_{4}
\end{pmatrix}=M_1M_2$ and $z_5=\det M_2$. We note that $\nu: Q_7\to Q_4$ is an $\SL_2-$torsor. Recall that the scheme $Q_4$ is pointed at $(0, 0, 0, 0, 0)$. 
Consider the associated $\A^1-$fiber sequence 
$$\SL_2\to Q_7\to Q_4$$ hence the long exact sequence of the $\#A^1$-homotopy sheaves 
$$\piA_3(\SL_2)\to  \piA_3(Q_7)\to \piA_3(Q_4)\to \piA_2(\SL_2) \to 0 .$$ 
Note that $\SL_2\simeq \A^2\setminus\{0\}$ hence by Proposition~\ref{van1}, $\coh^3(X, \piA_2(\A^2\setminus\{0\}))=0$.
Also note that by~\Cref{quadric-hom-type}(2), there is $\#A^1$-weak equivalence $Q_7\sim_{\#A^1}\#A^4\setminus\{0\}$, which induces isomorphism   $\KMW_4\simeq \piA_3(\#A^4\setminus\{0\})\simeq \piA_3(Q_7)$ hence we can deduce that the map $$\coh^3(X, \KMW_4)\to \coh^3(X, \piA_3(Q_4))$$ is surjective. 
    In the case $k=\-{\#F}_p$, the cohomology group  $\coh^3(X, \KMW_4)=0$ (see Proposition~\ref{van}), hence $ \coh^3(X, \piA_3(Q_4)) =0$.    
\end{proof}

\begin{proposition}\label{van}
Let $\mathbb{F}_p$ ($p$ a prime $\neq 2$) be the field with $p$ elements and let $\mathbb{F}_p \subset k \subset \overline{\mathbb{F}}_p$ be a field.
Let $A$ be a smooth affine algebra of dimension $d\geq 3$ over
$k$ and $X=\Spec(A)$. Then $\coh^{d}(X, K^{MW}_{d+1})=(0)$.    
\end{proposition}
\begin{proof}
    By~\cite[Theorem 4.9]{Fas11}, we have $\coh^{d}(X, K^{MW}_{d+1})\simeq Um_{d+1}(A)/E_{d+1}(A)$. 

{\bf Case 1:} Assume that $k$ is any finite field.

Then $A$ is a finitely generated ring. Hence by \cite[Theorem 18.2]{SuVa76}, we have $Um_{d+1}(A)/E_{d+1}(A)=(0)$.

{\bf  Case 2:} Assume $k=\overline{\mathbb{F}}_p$.
 
Let $\bar a =(a_1,\ldots,a_{d+1})$ be a unimodular vector in $Um_{d+1}(A)$. Then there exist $b=(b_1,\ldots,b_{d+1})$ such that $\Sigma_{i=1}^{d+1}a_ib_i=1$. 
Note that the vectors $a$ and $b$ are defined over $A'$ where $A'$ is a finitely generated algebra over a field $L$  such that $L$ is a finite extension of $\mathbb{F}_p$. Therefore $L$ is finite field and $A\cong A' \otimes_{L} k$.   Again by 
\cite[Theorem 18.2]{SuVa76}, there exists $\sigma \in E_{d+1}(A') \subset E_{d+1}(A)$ such that $\sigma (\bar a)=(1,0,\ldots,0)$.
Hence in this case also, we have $Um_{d+1}(A)/E_{d+1}(A)=(0)$.
\end{proof}

\begin{remark}
    We should note that in the setting of Proposition~\ref{vanishing-dim3}, for general algebraically closed field $k$, the argument of the proof of Proposition~\ref{vanishing-dim3} actually shows that the cohomology group $\coh^3(X, \piA_3(Q_4))$ is a quotient of the cokernel $$\Coker (\coh^3(X, \piA_3(\SL_2))\to \coh^3(X, \KMW_4)).$$ We already know that  $\coh^3(X, \KMW_4))$ is uniquely divisible from~\cite{Fas15}.
    If we know some torsion properties of this cokernel, we can conclude the desired vanishing of $\coh^3(X, \piA_3(Q_4))$ for any algebraically closed field.
\end{remark}

In the case $d=4$ we have the following vanishing result.
\begin{proposition}\label{vanishing-dim4}
Let $X$ be a smooth affine scheme of dimension $4$ over an algebraically closed field $k$ of characteristic $\neq 2$. Then $\coh^4(X, \piA_4(Q_{6}))=0.$
\end{proposition}
\begin{proof}
By~\cite[Theorem 4]{AF13} there is a surjection $$\coh^4(X, \piA_{3}(\A^{3}- 0))\to \coh^4(X, \piA_4(Q_{6})).$$
Since $k$ is an algebraically closed field, from ~\cite[proof of Theorem 2, subsection 6.4 on page 1058]{AF15}, we have
$$\coh^4(X, \piA_{3}(\A^{3}- 0))=0.$$\end{proof}

We now summarize the results of this section for the convenience of the reader.\begin{theorem} \label{vanishing}
Let $X$ be a smooth affine scheme of dimension $d$ over an algebraically closed field $k$ of characteristic $\neq 2$. Then
\begin{enumerate}
\item[$(1)$] if $d=2$,  $\coh^2(X, \piA_2(Q_{2}))=0.$ 

\item[$(2)$] if $d=3$, $\coh^3(X, \piA_3(Q_{4}))=0$ provided $k=\-{\F}_p$ (prime $p\neq 2$).

\item[$(3)$] if $d=4$, $\coh^4(X, \piA_4(Q_{6}))=0.$

\item[$(4)$] if $d\geq 5$, $\coh^d(X,\piA_d(Q_{2d-2}))=0,$ provided $char(k)=0$.
\item[$(5)$] if $d\geq 5$, then $\coh^{d}(X, \piA_{d}(Q_{2d-2}))$ is $d!\cdot (d-1)!\cdot (d-2)!-$torsion.
\item[$(6)$] if $d\geq 5$ and $d!\in k^{\x}$, then $\coh^{d}(X, \piA_{d}(Q_{2d-2}))$ is $ (d-1)!\cdot (d-2)!-$torsion.
\item[$(7)$] for $d\geq 5$, $\coh^d(X,\piA_d(Q_{2d-2}))=0,$ provided Conjecture~\ref{AF-conj} holds.
\end{enumerate}

\end{theorem}
\begin{proof}
(1) is Proposition~\ref{dim2}. (2) is Proposition~\ref{vanishing-dim3}. (3) is Proposition~\ref{vanishing-dim4}. (4), (5) and (6) are~\Cref{cor:vanishing} (1), (2) and (3) respectively. (7) is Proposition~\ref{Q-vanish-mod-conj}.
\end{proof}
\section{Cancellation and splitting of symplectic modules}\label{sec:cancel-and-split}

In this section, we prove the main results of the paper about the symplectic modules.  

We refer the reader to~\cite[Chapter VII, \S 5]{Lam06} for basics on symplectic modules/spaces over a commutative ring. For the convenience of the reader, we recall some basic notions relevant to our context.

\subsection{Symplectic groups and symplectic bundles}

Let us quickly recall the symplectic groups. Let $R$ be a commutative ring with unity and $\GL_n(R)$ the group of all invertible matrices over $R$.
Let $\psi_2 = \left(\begin{smallmatrix}0 & 1 \\ -1 & 0 \end{smallmatrix}\right)$ and let $\psi_{2n}=\perp_{1}^{n} \psi_2$ be the standard symplectic form of rank $2n$. Now we recall the definition of the even-order symplectic groups. For $n\geq 1$
$${\Sp}_{2n}(R)=\{M\in \GL_{2n}(R)\ | \ {^t  M}\psi_{2n} M=\psi_{2n}  \}.$$
This defines symplectic group schemes $\Sp_{2n}$ and a sequence of group schemes $\Sp_{2n-2}\to \Sp_{2n}$ defined by $A\mapsto \begin{pmatrix}
    A & 0\\
    0 & I_2
\end{pmatrix}$ 
the block sum of matrices where $I_2$ is the $2\x 2$-identity matrix.

\begin{definition}\label{def:symp-mod}

A symplectic $R-$module is a pair $(P, \phi)$ where $P$
is a finitely generated projective $R-$module and $\phi: P\times P \xrightarrow{} R$ is a bilinear map with $\phi(x,x)=0$ for all $x\in P$. We have an adjoint map
$adj: P \xrightarrow{} P^*:=\Hom(P,R)$
given by $x \mapsto \phi(x,-)$.
 We say that $P$ is non-degenerate if the adjoint map $adj$ is an isomorphism.

\end{definition}

In the paper, symplectic modules are always non-degenerate.

\subsection{Affine representability for symplectic bundles}\label{rep}
 Let $X$ be a $k$-scheme. For every integer $r\geq 1$, let $V_{r}^{\Sp}(X)$ be the set of isomorphism classes of rank $r$ symplectic vector bundles over $X.$ 
We recall the representability theorem from~\cite{AHW18} for symplectic vector bundles that will be used in~\Cref{sec:cancel-and-split}. We note that the theorem stated below is a special case of the more general result from~\cite{AHW18}. 
\begin{theorem}(\cite[Theorems 2.5.3, 3.3.3 and 4.1.2]{AHW18}, see also \cite[Remark 6.23]{Sch17}) \label{thm:rep:Sp}
   Let $X$ be a smooth affine scheme over a field $k$. 
Then for any even integer $r\geq 2$, there is a bijection of sets $$V_{r}^{\Sp}(X)\simeq [X, \BSp_{r}]_{\A^1}$$ 
which is functorial in $X$. 
\end{theorem}

Let $R$ be a commutative ring. We recall  
the \emph{symplectic hyperbolic plane} $\#H(R)$ over $R$ defined by rank 2 free $R$-module $R\oplus R$ endowed with the bilinear form associated to the $2\x 2$-matrix $\left(\begin{smallmatrix}0 & 1 \\ -1 & 0 \end{smallmatrix}\right)$.

\subsection{Cancellation of symplectic module}
We prove the following result about the cancellation of symplectic modules.

\begin{theorem}\label{cancel2n-2}
    Let $X=\Spec R$ be a smooth affine scheme of dimension $2n$ $(n\geq 2$) over an algebraically closed field $k$. Let $(2n)!$ be invertible in $k.$ Then the map 
    $$V_{2n-2}^{\Sp}(X)\to V_{2n}^{\Sp}(X)$$ is injective. In other words, if two symplectic rank $2n-2-$vector bundles $\@E$ and $\@E'$ are stably isomorphic, \ie $\@E \oplus \#H(R)\simeq \@E'\oplus \#H(R)$, then  $\@E \simeq \@E'$. 
\end{theorem}
\noindent
\textbf{Some preliminaries and computations.} 

To aid in the proof of the above theorem, we need to establish some vanishing result for a cohomology group of a homotopy sheaf. In order to do that,
let us first recall some notation from \cite[Section 8]{PD22}. 
We have the following commutative diagram 
\begin{equation}\label{eqn:fib-F}
\xymatrix{
\A^{2n}\setminus\{0\} \ar@{=}[r] \ar[d] & \A^{2n}\setminus\{0\}\simeq \Sp_{2n}/\Sp_{2n-2} \ar[r] \ar[d] &  \star\ar[d] \\
F'_{n-1} \ar[r] \ar[d] & \BSp_{2n-2} \ar[r] \ar[d] & \BSp \ar@{=}[d]\\
F'_n \ar[r] & \BSp_{2n} \ar[r] & \BSp 
}
\end{equation}
where all the rows and columns are  $\A^1-$fiber sequences, where $F'_{n-1}$ and $F'_n$ are the homotopy fibers of the naturals maps $\BSp_{2n-2} \rightarrow \BSp $ and $\BSp_{2n}  \rightarrow \BSp$ respectively. We note the following vanishing result.
\begin{proposition}
    Let $k$ be an algebraically closed field such that $(2n)!$ is invertible in $k$ ($n\geq 2$). Let $X$ be a smooth affine $k-$scheme of dimension $2n$. Then (in the notation above)
\begin{equation}\label{lem:vanish}
 \coh^{2n}(X, \piA_{2n}(F'_{n-1}))=0.   
 \end{equation}
 \end{proposition}

\begin{proof}
    From \cite[Section 8]{PD22},
we have the following computations of the $\A^1-$homotopy sheaves of $F'_{n-1}$ and $F'_n$.
\begin{align}\label{eq:homotopy gps of hom-fib}
    \piA_i(F'_{n-1})=\begin{cases}
        0  & \text{for} \ \ i\leq 2n-2\\
        \KMW_{2n} & \text{for} \ \ i=2n-1
    \end{cases}
\end{align}
and \begin{align*}
    \piA_i(F'_{n})=\begin{cases}
        0  & \text{for} \ \ i\leq 2n\\
        \KMW_{2n+2} & \text{for} \ \ i=2n+1.
    \end{cases}
\end{align*} Moreover, we have the following long exact sequence of $\A^1-$homotopy sheaves associated to the leftmost column fiber sequence in the diagram ~\eqref{eqn:fib-F}. 
$$ \cdots\to \piA_{2n+1}(F'_n)\to \piA_{2n}(\A^{2n}\setminus\{0\})\to \piA_{2n}(F'_{n-1}) \to \piA_{2n}(F'_n)=0$$
As a consequence, we have the following epimorphism of Nisnevich sheaves 
\begin{equation} \label{epi-pi}
    \piA_{2n}(\A^{2n}\setminus\{0\})\to \piA_{2n}(F'_{n-1}).
\end{equation}
 In particular, $\coh^{2n}(X, \piA_{2n}(\A^{2n}\setminus\{0\}))$ maps onto  $\coh^{2n}(X, \piA_{2n}(F'_{n-1}))$ since the Nisnevich cohomological dimension of $X$ is at most $2n$. From~\cite[Theorem 3]{Fas24}, $\coh^{2n}(X, \piA_{2n}(\A^{2n}\setminus\{0\}))=0$. (This is one place where the assumptions on the characteristic of the field $k$ come into play.) Hence, it follows that 
$ \coh^{2n}(X, \piA_{2n}(F'_{n-1}))=0.   $
\end{proof}

 Consider the Moore-Postnikov tower of the map $ \BSp_{2n-2}\to \BSp_{2n}$ and let $\BSp_{2n-2}[2n-1] $ be the $(2n-1)^{th}$ term of the tower. Then by~\Cref{MP-tower}(3), we have the following $\#A^1-$fiber sequence  $$ \K(\KMW_{2n}, 2n-1)  \to \BSp_{2n-2}[2n-1] \to \BSp_{2n}.$$
As noted in~\eqref{eqn:fib-F}, we have $\#A^1-$fiber sequence 
$$\#A^{2n}\setminus\{0\}\simeq \Sp_{2n}/\Sp_{2n-2} \to \BSp_{2n-2} \to \BSp_{2n}.$$
We have a commutative diagram 
 $$\xymatrix{
 \BSp_{2n-2} \ar[r] \ar[d] & \BSp_{2n} \ar@{=}[d]\\
 \BSp_{2n-2}[2n-1] \ar[r] & \BSp_{2n} 
}$$
By~\cite[Proposition 6.3.5]{Hov99} (dual statement for fiber sequences), we have the commutative diagram
$$\xymatrix{
\Omega \BSp_{2n} \ar[r]^-{\partial} \ar@{=}[d]  \ar@{=}[d] & \#A^{2n}\setminus\{0\} \ar[r] \ar[d]_-{\tau'} & \BSp_{2n-2} \ar[r] \ar[d]^{i_{2n-1}} & \BSp_{2n} \ar@{=}[d] \\
\Omega \BSp_{2n} \ar[r]^-{\Omega b_n}  & \K(\KMW_{2n}, 2n-1)  \ar[r] & \BSp_{2n-2}[2n-1] \ar[r]  & \BSp_{2n}
}$$
 By~\cite[Proposition 6.5.3]{Hov99}, for each $U\in \Sm/k$, the induced map 
$$\partial^*: [U, \Sp_{2n}]_{\#A^1}\simeq [U, \Omega \BSp_{2n}]_{\#A^1} \to [U,\#A^{2n}\setminus\{0\}]_{\#A^1}$$ is given by the action of an element $g\in \Sp_{2n}(U)$ on the base point $(1, 0,\dots, 0)\in (\#A^{2n}\setminus\{0\})(U)$ \ie assigning $g$ to its first row  $[1, 0, \dots, 0]\cdot g$. 
Thus, we can identify the map $$\partial: \Sp_{2n}\simeq \Omega \BSp_{2n} \to \#A^{2n}\setminus\{0\} $$ as the projection
 $$pr_1: \Sp_{2n}\to \#A^{2n}\setminus\{0\}$$ given by $g\mapsto [1, 0, \dots, 0]\cdot g$ sending a $2n\x 2n$-square matrix $g$ to its first row.   

Hence we have the commutative diagram
$$\xymatrix{
\Sp_{2n} \ar[r]^-{pr_1} \ar@{=}[d]  \ar@{=}[d] & \#A^{2n}\setminus\{0\} \ar[r] \ar[d]_-{\tau'} & \BSp_{2n-2} \ar[r] \ar[d]^{i_{2n-1}} & \BSp_{2n} \ar@{=}[d]\\
\Sp_{2n} \ar[r]^-{\Omega b_n}  & \K(\KMW_{2n}, 2n-1)  \ar[r] & \BSp_{2n-2}[2n-1] \ar[r]  & \BSp_{2n}}$$
Thus, $$\Omega b_{n} = \tau'\circ pr_1: \Sp_{2n} \to  \K(\KMW_{2n}, 2n-1)$$
in $\@H_{\Dot}(k).$
                                                           Hence \begin{equation}\label{eq:Omega-tau}
    \Omega b_{n}(\beta) = \tau' (pr_1(\beta)), \ \ \forall \ \beta \in [X, \Sp_{2n}]_{\#A^1}.
\end{equation}
We note the following identities.
\begin{proposition} \label{prop:Omegauptosign} Let $X$ be a smooth $k-$scheme of dimension $2n$ over an algebraically closed field $k$ of characteristic $\neq 2.$  Then under the reduction map $$\xymatrix{
 \coh^{2n-1}(X, \KMW_{2n})\ar[r]^{\tau} \ar[d]^{\simeq} & \coh^{2n-1}(X, \KM_{2n})\ar[d]^{\simeq} \\
[X, \K(\KMW_{2n}, 2n-1)]_{\#A^1}\ar[r]^{\tau} & [X, \K(\KM_{2n}, 2n-1)]_{\#A^1}
}$$ we have for $\beta \in [X, \Sp_{2n}]_{\#A^1}$
$$\tau\circ \tau' \circ pr_1(\beta)=\tau\circ (\Omega b_{n})(\beta)=\pm(\Omega c_{2n})(\beta).$$
   \end{proposition}
   \begin{proof}
       The first equality is established in~\eqref{eq:Omega-tau}.        
For the second equality consider the commutative diagram $$\xymatrix{
 \Sp_{2n-2} \ar[r] \ar[d] & \Sp_{2n} \ar[d]\\
 \SL_{2n-2}\ar[r] & \SL_{2n} 
}$$ hence the commutative diagram
$$\xymatrix{
 \BSp_{2n-2} \ar[r] \ar[d] & \BSp_{2n} \ar[d]\\
 \BSL_{2n-2}\ar[r] & \BSL_{2n} 
}$$
hence by naturality of $k-$invariants we have the commutative diagram  
\begin{equation}\label{diag:euler-borel}
    \xymatrix{
\BSp_{2n} \ar[r]^-{b_n} \ar[d] & \K(\KMW_{2n}, 2n)\ar@{=}[d]\\
\BSL_{2n} \ar[r]^-{e_{2n}}& \K(\KMW_{2n}, 2n)
}\end{equation} where $e_{2n} \in \coh^{2n}(\BSL_{2n}, \KMW_{2n})\simeq [\BSL_{2n}, \K(\KMW_{2n}, 2n)]_{\#A^1}$ is the ($k$-invariant) universal Euler class and $b_{n} \in \coh^{2n}(\BSp_{2n}, \KMW_{2n})$ is the universal Borel class. That the diagram commutes follows from~\cite[Proposition 4.3(ii)]{HW19}. In the loc. cit. Pontryagin class $p_n(\Sp_{2n})$ is called Borel class in~\cite{PW21}. By \cite[Theorem 1 and Proposition 5.8]{AF16a}, the $k$-invariant Euler class is the universal Chow-Witt Euler class up to a unit of $\KMW_0(k)=\#Z$ (note $k$ is an algebraically closed field) and under the reduction map $\tau: \coh^{2n-1}(X, \KMW_{2n})\to \coh^{2n-1}(X, \KM_{2n}) $, the Euler class maps to the Chern class $c_{2n}$. Thus, we have 
$$\tau\circ (\Omega b_{n})(\beta)= \tau(\Omega e_{2n} (\beta))=\pm(\Omega c_{2n})(\beta)$$
where $\beta \in [X, \Sp_{2n}]$ and $c_{2n}$ is the universal Chern class in $\coh^{2n}(\BSL_{2n}, \KM_{2n})$. 
  \end{proof}

 We now proceed with the proof of the Theorem~\ref{cancel2n-2}.
 \begin{proof}[\bf Proof of~\Cref{cancel2n-2}:]
By the representability result in Subsection~\ref{rep}, we have the commutative diagram 
$$\xymatrix{
V_{2n-2}^{\Sp}(X)\ar[d]^{\simeq} \ar[r] &  V_{2n}^{\Sp}(X)\ar[d]^{\simeq} \\
[X, \BSp_{2n-2}]_{\A^1} \ar[r] & [X, \BSp_{2n}]_{\A^1}
}$$ Thus to show the injectivity of the top horizontal arrow it suffices to show the injectivity of the bottom horizontal map. We show this by analyzing the Moore-Postnikov tower (see Proposition~\ref{MP-tower} for the notation) for the $\A^1-$fiber sequence 
    $$F'_{n-1} \to \BSp_{2n-2} \to \BSp.$$ Note that $\piA_1(\BSp)=0$.
    Recall from Proposition~\ref{MP-tower} that there is a sequence of $\#A^1-$fibrations
    $$\cdots \to E_r\to E_{r-1}\to\cdots\to E_2\to E_1\to E_0=\BSp$$
    such that 
    $\BSp_{2n-2} \simeq \hlim_r E_r$ and for each $r\geq 1$, there is an $\#A^1-$fiber sequence 
    $$\K(\piA_r(\BSp_{2n}), r)\to E_r\xrightarrow{q_r} E_{r-1}.$$
    Since as noted in~\eqref{eq:homotopy gps of hom-fib}, the first non-trivial $\#A^1-$homotopy sheaf of the homotopy fiber $F'_{n-1}$ is in the degree $2n-1$, the non-trivial part of the Moore-Postnikov tower is 
    $$\xymatrix{
   \K(\piA_{2n}(F'_{n-1}), 2n) \ar[r]  & E_{2n} \ar[d] &\\
    \K(\KMW_{2n}, 2n-1) \ar[r] & E_{2n-1}\ar[d]^q & \\ 
     & E_{2n-2}=\BSp \ar[r] & \K(\KMW_{2n}, 2n)
    }$$
By the Moore-Postnikov tower associated to the map $\BSp_{2n}\to \BSp$ and since $\dim X=2n$, we can deduce that $[X, \BSp_{2n-2}]_{\A^1}\to [X, \BSp]_{\A^1}$ is bijective. Consider the diagram      $$\xymatrix{
[X, \BSp_{2n-2}]_{\A^1}\ar[r]\ar[d] & [X, \BSp]_{\A^1}\\
[X, E_{2n}]_{\A^1} \ar[r] & [X, E_{2n-1}]_{\A^1} \ar[u] 
}$$ 
where the leftmost vertical map is bijective, since $\dim X=2n$ (\cite[Corollary B.5]{Mor12}).
Hence to prove the injectivity of the top horizontal map in the above diagram, it is sufficient to prove that 
\begin{enumerate}
     \item the map $[X, E_{2n}]_{\A^1}\to [X, E_{2n-1}]_{\A^1}$ is injective and \label{part1}
      \item the map $[X, E_{2n-1}]_{\A^1}\to [X, \BSp]_{\A^1}$ is injective. \label{part2}

\end{enumerate}
\_{For part~\eqref{part1}:} we have the following $\A^1-$fiber sequence $$\K(\piA_{2n}(F'_{n-1}), 2n)\to E_{2n}\to E_{2n-1}$$
hence the following exact sequence 
$$\coh^{2n}(X, \piA_{2n}(F'_{n-1}))\to [X, E_{2n}]_{\A^1}\to [X, E_{2n-1}]_{\A^1}.$$
By Proposition \ref{lem:vanish}, $\coh^ {2n}(X, \piA_{2n}(F'_{n-1}))=0$, hence the map 
    $$[X, E_{2n}]_{\A^1}\to [X, E_{2n-1}]_{\A^1}$$ is injective.\\
\_{For part~\eqref{part2}:} consider the following  $\A^1-$fiber sequence 
    $$ E_{2n-1}\xrightarrow{q} \BSp \xrightarrow{p} \K(\KMW_{2n}, 2n)$$
Let $\xi \in [X, E_{2n-1}]_{\#A^1}$. Since $E_{2n-1}$ is fibrant, $\xi$ is represented by $\xi\in \Rmap_*(X_+, E_{2n-1})_0$ (in the notation of~\Cref{sec:fib-seq}). Applying $\Rmap_*(X_+, -)$ to the above fiber sequence and by \cite[Lemma 2.2.2]{Fas24} we have the fiber sequence   
$$\Rmap_*(X_+, \K(\KMW_{2n}, 2n-1)) \to \Rmap_*(X_+, E_{2n-1}) \xrightarrow{q_*} \Rmap_*(X_+, \BSp)$$ 
of pointed sets. Since $\BSp$ is an abelian group object in $\@H_*^{\#A^1}(k)$ as noted in~\cite[Proposition 8.1]{PD22}, hence applying the observation in~\Cref{sec:fib-seq} yields the following long exact sequence 
$$[X, \Sp]_{\#A^1}\simeq [X,\Omega \BSp]_{\#A^1}\xrightarrow{\Delta} \coh^{2n-1}(X, \KMW_{2n})\to [X, E_{2n-1}]_{\A^1}\xrightarrow{q_*} [X, \BSp]_{\A^1}.$$
Let $\xi$ and $\xi'$ be two elements in $[X, E_{2n-1}]_{\A^1}$ such that $q_*(\xi)=q_*(\xi')$. Then by the exactness of homotopy groups (and the pointed sets), there is $g\in \coh^{2n-1}(X, \KMW_{2n})$ such that $g\cdot\xi=\xi'.$ We point the sets at $\xi$ so that   
we have the long exact sequence 
$$[X, \Sp]_{\#A^1}\simeq [X,\Omega \BSp]_{\#A^1}\xrightarrow{\Delta(b_n, \xi)} \coh^{2n-1}(X, \KMW_{2n})\to [X, E_{2n-1}]_{\A^1}\to [X, \BSp]_{\A^1}$$
where the image of $\Delta$ is the stabilizer of $\xi$. Thus, to show $\xi=\xi'$ it suffices to show that $g $ belongs to the stabilizer of $\xi$, \ie $g$ is in the image of $\Delta(b_n, \xi).$
We prove this claim by adapting the arguments as in \cite[section 8, pages 14859-14861]{PD22} to our current setting. 

In several steps we prove the claim that $g\in \coh^{2n-1}(X, \KMW_{2n})$ is in the image of the map $$\Delta(b_n, \xi): [X, \Sp]_{\#A^1}\to  \coh^{2n-1}(X, \KMW_{2n}).$$
\noindent
\_{\bf Step 1:} By \cite[Corollary 3.0.2]{Fas24}, there is an isomorphism $\KMW_{2n}|_X\simeq\KM_{2n}|_X$ of sheaves ($\dim X=2n$ and base field $k$ is algebraically closed of characteristic $\neq 2$). Hence, it follows that the map $$\tau: \coh^{2n-1}(X, \KMW_{2n})\to  \coh^{2n-1}(X, \KM_{2n})$$ is an isomorphism. Thus to show $g\in \im (\Delta(b_n, \xi))$, it suffices to show that 
 $$\tau(g)\in \im (\tau\circ \Delta(b_n, \xi)).$$ 
\noindent
\_{\bf Step 2:} 
 We have by \cite[Proposition 6.3 and Theorem 8.2(42)]{PD22}, (replacing $GL$ by $\Sp$) and under the assumption $b_n(\xi)=0$, for $\beta\in [X, \Sp]_{\#A^1}$
\begin{equation}\label{eq:sum}
\Delta(b_n, \xi)(\beta)=(\Omega b_n)(\beta)+ \sum_{r=1} ^{n-1} (\Omega b_r)(\beta)\cdot b_{n-1-r}(\beta).
\end{equation}

Consider the commutative diagram in $\@H^{\A^1}_{\Dot}(k)$ (cf. \cite[page 14860]{PD22}).
\begin{equation}\label{diag:suslin}
    \xymatrix{
 & & \SL_{2n} \ar[r]^H  \ar[d] & \Sp_{4n} \ar[d] \ar[r] & \Sp \simeq \Omega \BSp \ar[r]^-{\Omega b_r} \ar[ld]^F & \K(\KMW_{2r}, 2r-1)\\
X\ar[r]^-{(a, b)} & Q_{4n-1}\ar[ru]^{\beta_{2n}} \ar[r]^{\alpha_{2n}} & \SL_{2^{2n-1}} \ar[r] & \SL & &  
}\end{equation}
where $H: \SL_{2n} \to  \Sp_{4n} $ is given by $A\mapsto \begin{pmatrix}
    A & 0 \\
    0 & (A^T)^{-1}
\end{pmatrix}.$ 
Here $\alpha_{2n}: Q_{4n-1} \to  \SL_{2^{2n-1}} $ is given by Suslin matrices (see ~\Cref{subsec:sus-matrices} for notation and \cite[\S5]{Suslin77} and \cite[Page 14836]{PD22} for further details) and $\beta_{2n}: Q_{4n-1} \to  \SL_{2n} $ is a lift of $\alpha_{2n}$ up to homotopy. Also recall that $(a, b)$ denotes the $k-$morphism of schemes $X\to Q_{4n-1}$ (as discussed under the definition of quadrics at the beginning of~\Cref{subsec:sus-matrices}).

Note that $ (\Omega b_r)(H\beta_{2n})\in \coh^{2r-1}(Q_{4n+1}, \KMW_{2r}).$
Recall that by \cite[Proposition 1.1.5(3)]{AF22}, we have

\begin{align*}
\coh^{2r-1}(Q_{4n-1}, \KMW_{2r})=&   \begin{cases}
       0 \ & \text{for} \ 2r-1\neq 2n-1 \\
       \KMW_0(k) \ &\text{for} \ 2r-1= 2n-1 \ (r=n).
    \end{cases}
\end{align*}
Hence, combining this with~\eqref{eq:sum}, it follows that  
\begin{equation}\label{eq:red}
    \Delta(b_n, \xi)(H[(\beta_{2n})])= (\Omega b_n)(H(\beta_{2n})) 
\end{equation}
in the cohomology group  $\coh^{2n-1}(Q_{4n-1}, \KMW_{2n}).$

\noindent
\_{\bf Step 3:}
In this step we show that in~\eqref{eq:red}, after applying the reduction map $\tau$ to the~\eqref{eq:red} simplifies as $$ \tau\circ (\Omega b_n)\circ H=(\Omega c_{2n})\circ F\circ H.$$
To see this, consider the commutative diagram induced by~\eqref{diag:euler-borel} $$\xymatrix{
\BSp \ar[r]^-{b_n} \ar[d]^F & \K(\KMW_{2n}, 2n) \ar[d]^{\tau}  \\
\BSL \ar[r]^-{c_{2n}} & \K(\KM_{2n}, 2n)  
}$$ which induces the commutative diagram 
$$\xymatrix{
\Sp\simeq \Omega \BSp \ar[r]^-{\Omega b_n} \ar[d]^F & \K(\KMW_{2n}, 2n-1) \ar[d]^{\tau} \\
\SL\simeq \Omega \BSL \ar[r]^-{\Omega c_{2n}} & \K(\KM_{2n}, 2n-1)  
}$$
Hence we get
\begin{equation}\label{eq:borel-chern}
    \tau\circ (\Omega b_n)\circ H=(\Omega c_{2n})\circ F\circ H: \SK_1(X)=[X, \SL]_{\A^1}\to \coh^{2n-1}(X, \KM_{2n}).
    \end{equation}

    \noindent
\_{\bf Step 4:} 
Let $F: \Sp\to \SL$ be the forgetful map and consider the composite map
$$F\circ H: \SK_1(X)=[X, \SL]_{\#A^1}\xrightarrow{H}  [X, \Sp]_{\#A^1}  \xrightarrow{F} [X, \SL]_{\#A^1}.$$ 
In this step we show that for each $(a, b)\in \Hom(X, Q_{4n-1})$
given by the tuple $(a, b)$,  $$F\circ H(\beta_{2n}(a, b))=2[\beta_{2n}(a, b)].$$

To see this, note that by \cite[Lemma 5.3]{Suslin77} (applied to $r=2n-1$ in loc. cit. so $r$ is of the form $4k+3$ or $4k+1$ in Lemma 5.3 in loc. cit.), we have $$[\alpha_{2n}(a, b)]=[(\alpha_{2n}(a, b)^T)^{-1}]$$
Consider $$F\circ H(\alpha_{2n}(a, b))=\bigg[\begin{pmatrix}
    \alpha_{2n}(a, b) & 0\\
    0 & (\alpha_{2n}(a, b)^T)^{-1}
\end{pmatrix}\bigg]=[\alpha_{2n}(a, b)]+[(\alpha_{2n}(a, b)^T)^{-1}]=2[\alpha_{2n}(a, b)].$$
Note that $[\alpha_{2n}(a, b)]=[\beta_{2n}(a, b)]$ in $\SK_1(X)$.  As a consequence
\begin{equation}\label{eq:FH=2}
    F\circ H(\beta_{2n}(a, b))=2[\beta_{2n}(a, b)].
\end{equation}

\noindent
\_{\bf Step 5:} 
Recall the commutative diagram from the proof of~\Cref{prop:Omegauptosign}
$$\xymatrix{
\Sp_{2n} \ar[r]^-{pr_1} \ar@{=}[d]  \ar@{=}[d] & \#A^{2n}\setminus\{0\} \ar[r] \ar[d]_-{\tau'} & \BSp_{2n-2} \ar[r] \ar[d]^{i_{2n-1}} & \BSp_{2n} \ar@{=}[d]\\
\Sp_{2n} \ar[r]^-{\Omega b_n}  & \K(\KMW_{2n}, 2n-1)  \ar[r] & \BSp_{2n-2}[2n-1] \ar[r]  & \BSp_{2n}}$$
For the map $\tau'$ from the above diagram, we claim that the induced map 
  $$\tau'_*:[X,\A^{2n}\setminus \{0\}]_{\A^1}\to [X, \K(\KMW_{2n}, 2n-1)]_{\A^1}\simeq  \coh^{2n-1}(X, \KMW_{2n})$$
is surjective, under the assumption that $\dim X=2n$. 

To see this, by applying $[X,-]$ to the above diagram, consider the commutative diagram of long exact sequences of pointed sets 
$$\xymatrix{
[X, \Sp_{2n}]_{\#A^1} \ar[r] \ar@{=}[d]  \ar@{=}[d] & [X, \#A^{2n}\setminus\{0\}]_{\#A^1} \ar[r] \ar[d]_-{\tau'_*} & [X, \BSp_{2n-2}]_{\#A^1} \ar[r] \ar[d]^{{i_{2n-1}}_*} & [X, \BSp_{2n}]_{\#A^1} \ar@{=}[d]\\
[X, \Sp_{2n}]_{\#A^1} \ar[r]  & [X, \K(\KMW_{2n}, 2n-1)]_{\#A^1}  \ar[r] & [X, \BSp_{2n-2}[2n-1]]_{\#A^1} \ar[r]  & [X, \BSp_{2n}]_{\#A^1}
}$$ 
By~\cite[Appendix B]{Mor12} (for a detailed proof see~\cite[Proposition 6.2]{AF14}) it follows that the map  $${i_{2n-1}}_* : [X, \BSp_{2n-2}]_{\#A^1} \to [X, \BSp_{2n-2}[2n-1]]_{\#A^1}  $$ is surjective, as $\dim X=2n$. By diagram chase it follows that the map $$\tau'_*:[X,\A^{2n}\setminus \{0\}]_{\A^1}\to [X, \K(\KMW_{2n}, 2n-1)]_{\A^1}\simeq  \coh^{2n-1}(X, \KMW_{2n})$$
is surjective.

 Finally, we have all the ingredients to finish the proof.
 
 \noindent
\_{\bf Step 6:} Recall that by Step 1 we reduced to show that 
 $$\tau(g)\in \im (\tau\circ \Delta(b_n, \xi)).$$  
\begin{alignat*}{3}
    \tau(g) = & \pm 2(2n-1)! \tau(\alpha) &&{\text{\small(as $\coh^{2n-1}(X, \KM_{2n})$ is $(2n-1)!$-divisible~\cite[Lem. 4.0.3]{Fas24})}} \\
    = &  \pm 2(2n-1)! \tau(\tau'_*([(a, b)]))  &&\text{(since $\tau'_*$ is surjective by {\bf Step 5})}\\
    = & \pm2\tau\bigg(\tau'_*\bigg(\dfrac{(2n-1)!}{2}h([(a, b)])\bigg)\bigg)  &&\text{(as $h=2$ in $\KMW_0(k)$ as $k$ is algebraically closed)} \\
    = & \pm2 \tau(\tau'_*( {\psi_{2n}}_*([a, b])))  &&\text{(follows by~\Cref{prop:suslin-mult}(2) as $\dim X=2n$ and $(2n)!\in k^\x$)}\\
    = & \pm 2\tau(\tau'_*([a_0^{(2n-1)!}, a_1, \cdots, a_{2n-1}])) &&\text{(by definition of ${\psi_{2n}}_*$)}\\
    = & \pm 2\tau(\tau'_*(pr_1(\beta_{2n}([a, b])))) \\
    = & \pm 2(\Omega c_{2n})(\beta_{2n}([a, b]))  &&\text{(this follows from Proposition~\ref{prop:Omegauptosign})}\\
    = & \pm 2(\Omega c_{2n})\circ F\circ H(\beta_{2n}([a, b]))  &&\text{(by \eqref{eq:FH=2} in {\bf Step 4})}\\
    = &  \pm \tau\circ(\Omega b_n)(H(\beta_{2n}([a, b])))  &&\text{(by \eqref{eq:borel-chern} in {\bf Step 3})} \\
    = & \pm \tau(\Delta(b_n, \xi)(H\beta_{2n}([a, b])))  &&\text{(by \eqref{eq:red} in {\bf Step 2}).} 
\end{alignat*}

Hence $g$ is in the image of $\Delta$, thus $g$ is a stabilizer of $\xi$. Therefore we have $\xi=g\cdot \xi=\xi'.$   
Finally, we get that 
$$[X, E_{2n-1}]_{\A^1}\to [X, \BSp]_{\A^1}$$ is injective.     
\end{proof}
\begin{remark} In this remark we present an example that will establish that our ~\Cref{cancel2n-2}
is optimal. In fact, here, we will observe that Mohan Kumar's counter example \cite{MK85} with suitable modification works here as well.
Details are as follows.
  
   Consider for each even $r\geq 2$ the map
$$\phi_r: V_{r}^{\Sp}(X) \xrightarrow{} V_{r+2}^{\Sp}(X)$$ which sends $$(V,\psi_r) \mapsto (V\oplus \#H(R), \psi_r \perp \psi_2).$$ 
 Let $p$ be an odd prime number. Let $k$ be an algebraically closed field of characteristic $\neq 2.$ By \cite[Theorem 6.1, Corollary 6.2]{Wendt21} there exists $X=\Spec(R)$ a smooth affine variety of dimension $p+2$ over an algebraically closed field $k$ and a non-trivial stably free $R$-module $P$ of rank $p-1$ with a symplectic structure. In particular, this means that for such an $X$, the map $$ \phi_{p-1}: V_{p-1}^{\Sp}(X) \xrightarrow{} V_{p+1}^{\Sp}(X)$$  
is not injective. 
 
By \Cref{cancel2n-2}, applied to $X\x \#A^1_k$ where $X$ as in the above paragraph, the map $$V_{p+1}^{\Sp}(X\x \#A^1_k) \to V_{p+3}^{\Sp}(X\x \#A^1_k)$$ is injective. However, as noted above and the $\#A^1$-invariance of symplectic bundles on smooth affine $k$-schemes by~\cite[Theorem 3.3.3]{AHW18}, the map  $$ \phi_{p-1}: V_{p-1}^{\Sp}(X\x \#A^1_k) \xrightarrow{} V_{p+1}^{\Sp}(X\x \#A^1_k)$$  
is not injective. Hence, the cancellation result in~\Cref{cancel2n-2} is optimal as $\dim X\x \#A^1_k=p+3$.
 \end{remark}

\subsection{Splitting of symplectic module}

We prove the following result about the splitting of a symplectic module.
\begin{theorem}\label{thm:splitting}
Let $R$ be a smooth affine algebra over a perfect field $k$ and let $X=\Spec (R)$.  Let $P$ be a symplectic $R-$module of rank $2n$ ($n\geq 1$).

\begin{enumerate}
    \item Assume that $\dim(R)=2n$. Then $P\cong Q \oplus \#H(R) $ if and only if the Euler class $e_{2n}(P)=0$ in $\widetilde{\CH}^{2n}(X)$.

\item Assume that $\dim(R)=2n+1$, and $\coh^{2n+1}(X, \piA_{2n}(\mathbb{A}^{2n}\setminus\{0\}))=0.$
Then $P\cong Q \oplus \#H(R) $ if and only if the Euler class $e_{2n}(P)=0$ in $\widetilde{\CH}^{2n}(X)$.

Moreover, if the base field $k$ is algebraically closed of characteristic 0, then the splitting criterion holds.
\end{enumerate} 

\end{theorem}

\begin{proof}
\noindent
\begin{enumerate}
    \item It is clear that if $P\cong Q \oplus \#H(R)$, then the Euler class $e_{2n}(P)=0$ in $\widetilde{\CH}^{2n}(X)$.

    Conversely, let us assume $e_{2n}(P)=0$ in $\widetilde{\CH}^{2n}(X)$.  Let $V_{2n}^{\Sp}(X)$ denote the isomorphism class of rank $2n$ symplectic bundles over $X$. As noted in~\Cref{thm:rep:Sp}, we have $V_{2n}^{\Sp}(X) \cong [X, \BSp_{2n}]_{\mathbb{A}^1}$. Hence the projective module $P$ corresponds to a map $\xi: X\to \BSp_{2n}$ in the homotopy category.
  We need to show that if $e_{2n}(P)=0$ in $\widetilde{\CH}^{2n}(X)$ then the associated map $\xi: X\to \BSp_{2n}$ lifts to a map $\xi': X\to \BSp_{2n-2}$ upto homotopy.  
    $$\xymatrix{ 
    & \BSp_{2n-2} \ar[d] \\
     X \ar@{-->}[ru]^{\xi'} \ar[r]^{\xi} & \BSp_{2n}
    }$$ 
    Consider the Moore-Postnikov tower for the fiber sequence $$F'_{n-1} \to \BSp_{2n-2} \to \BSp.$$ 
 $$\xymatrix{
   \K(\piA_{2n}(F'_{n-1}), 2n) \ar[r]  & E_{2n} \ar[d]^{q_{2n}} &\\
    \K(\KMW_{2n}, 2n-1) \ar[r] & E_{2n-1}\ar[d]^{q_{2n-1}} \ar[r] & \K(\piA_{2n}(F'_{n-1}), 2n+1)\\ 
     & \BSp\simeq E_{2n-2} \ar[r] & \K(\KMW_{2n}, 2n)
    }$$
The above tower yields the following two fiber sequences
\begin{equation}\label{fib-1}
    E_{2n} \xrightarrow{q_{2n}} E_{2n-1} \xrightarrow{} \K(\piA_{2n}(F'_{n-1}), 2n+1)
\end{equation}
\begin{equation}\label{fib-2}
      E_{2n-1} \xrightarrow{q_{2n-1}} E_{2n-2} \xrightarrow{} \K(\KMW_{2n}, 2n)
\end{equation}
From (\ref{fib-1}), we have the following exact sequence 
$$[X, E_{2n}]_{\#A^1} \xrightarrow{(q_{2n})_*} [X, E_{2n-1}]_{\#A^1} \xrightarrow{} [X, \K(\piA_{2n}(F'_{n-1}), 2n+1)]_{\#A^1}$$
Note that $[X, \K(\pi_{2n}(F'_{n-1}), 2n+1)]_{\#A^1} \cong \coh^{2n+1}(X, \pi_{2n}(F'_{n-1}))$ which is a trivial group, since $2n+1>2n=\dim X$. Hence $(q_{2n})_{*}$ is surjective.

From (\ref{fib-2}), we have the following exact sequence
$$[X, E_{2n-1}]_{\#A^1} \xrightarrow{(q_{2n-1})_*} [X, E_{2n-2}]_{\#A^1} \xrightarrow{} [X, \K(\KMW_{2n}, 2n)]_{\#A^1}$$
The diagram $$\xymatrix{
[X, E_{2n-2}]_{\#A^1} \ar[r] \ar[d]^{\simeq} & [X, \K(\KMW_{2n}, 2n)]_{\#A^1}\ar@{=}[d]\\
[X, \BSp_{2n}]_{\#A^1}\ar[r]^-{b_n} & [X, \K(\KMW_{2n}, 2n)]_{\#A^1}
}$$ commutes such that the map 
$$[X, \BSp_{2n}]_{\#A^1}\xrightarrow{b_n} [X, \K(\KMW_{2n}, 2n)]_{\#A^1}$$
is taking the $n^{th}$ Borel class $b_n(\xi)$ which agrees with the Euler class $e_{2n}(P)$ (see~\cite[page 14858 (40)]{PD22}). Note that we have $[X, \K(\KMW_{2n}, 2n)]\cong \coh^{2n}(X,\KMW_{2n})$ which is isomorphic to $\widetilde{\CH}^{2n}(X)$.

The Euler class $e_{2n}(P)=0$ in  $\widetilde{\CH}^{2n}(X)$ \ie $b_n(\xi)=*$ implies that  the sequence of pointed sets $$ [X, E_{2n-1}]_{\#A^1} \xrightarrow{(q_{2n-1})_*} [X, E_{2n-2}]_{\#A^1} \xrightarrow{b_n} [X, \K(\KMW_{2n}, 2n)]_{\#A^1}$$ is exact. Hence there is $a\in [X, E_{2n-1}]_{\#A^1}$ such that $(q_{2n-1})_*(a)=\xi.$ As observed above the map 
$$[X, E_{2n}]_{\#A^1} \xrightarrow{(q_{2n})_*} [X, E_{2n-1}]_{\#A^1} $$
is surjective hence there is $a'\in [X, E_{2n}]$ such that $(q_{2n})_*(a')=a$. Furthermore, since the $\dim X=2n$ the maps $$ [X, \BSp_{2n-2}]\to [X, E_{r}]_{\#A^1}\to [X, E_{2n}]_{\#A^1}$$ are surjective for $r\geq 2n$. Hence the preimage $\xi'\in [X, \BSp_{2n-2}]_{\#A^1}$ of $a'$ is the desired lift of the map $\xi\in [X, \BSp_{2n}]_{\#A^1}.$

\item  We follow the same arguments as in part $(1)$ with suitable modification. 

We need to show that if $e_{2n}(P)=0$ in $\widetilde{\CH}^{2n}(X)$ then the associated map $\xi: X\to \BSp_{2n}$ lifts to a map $\xi': X\to \BSp_{2n-2}$ upto homotopy \ie the following diagram commutes in $\@H^{\#A^1}(k)$
    $$\xymatrix{ 
    & \BSp_{2n-2} \ar[d] \\
     X \ar@{-->}[ru]^{\xi'} \ar[r]^{\xi} & \BSp_{2n}
    }$$ 
    Consider the Moore-Postnikov tower for the fiber sequence $$F'_{n-1} \to \BSp_{2n-2} \to \BSp.$$ 
 $$\xymatrix{
   \K(\piA_{2n}(F'_{n-1}), 2n) \ar[r]  & E_{2n} \ar[d]^{q_{2n}} &\\
    \K(\KMW_{2n}, 2n-1) \ar[r] & E_{2n-1}\ar[d]^{q_{2n-1}} \ar[r] & \K(\piA_{2n}(F'_{n-1}), 2n+1)\\ 
     & \BSp\simeq E_{2n-2} \ar[r] & \K(\KMW_{2n}, 2n)
    }$$
The above tower yields the following two fiber sequences 
\begin{equation}\label{fib-3}
    E_{2n} \xrightarrow{q_{2n}} E_{2n-1} \xrightarrow{} \K(\piA_{2n}(F'_{n-1}), 2n+1)
\end{equation}
\begin{equation}\label{fib-4}
      E_{2n-1} \xrightarrow{q_{2n-1}} E_{2n-2} \xrightarrow{} \K(\KMW_{2n}, 2n)
\end{equation}
From (\ref{fib-3}), we have the following exact sequence 
$$[X, E_{2n}]_{\#A^1} \xrightarrow{(q_{2n})_*} [X, E_{2n-1}]_{\#A^1} \xrightarrow{} [X, \K(\piA_{2n}(F'_{n-1}), 2n+1)]_{\#A^1}$$
Note that $[X, \K(\pi_{2n}(F'_{n-1}), 2n+1)]_{\#A^1} \cong \coh^{2n+1}(X, \pi_{2n}(F'_{n-1}))$. 
By~\eqref{epi-pi}, there is an epimorphism of Nisnevich sheaves 
$$\piA_{2n}(\mathbb{A}^{2n}\setminus\{0\}) \to \piA_{2n}(F'_{n-1}).$$ 
This induces a surjective map of abelian groups 
$$\coh^{2n+1}(X, \piA_{2n}(\mathbb{A}^{2n}\setminus\{0\})) \xrightarrow{} \coh^{2n+1}(X, \piA_{2n}(F'_{n-1}))$$ since $\dim X=2n+1$.
Now  by hypothesis $\coh^{2n+1}(X, \piA_{2n}(\mathbb{A}^{2n}\setminus\{0\}))=0$, hence we have $\coh^{2n+1}(X, \pi_{2n}(F'_{n-1}))=0$.
Hence $(q_{2n})_{*}$ is surjective.

From (\ref{fib-4}), we have the following exact sequence
$$[X, E_{2n-1}]_{\#A^1} \xrightarrow{(q_{2n-1})_*} [X, E_{2n-2}]_{\#A^1} \xrightarrow{} [X, \K(\KMW_{2n}, 2n)]_{\#A^1}$$
The diagram $$\xymatrix{
[X, E_{2n-2}]_{\#A^1} \ar[r] \ar[d]^{\simeq} & [X, \K(\KMW_{2n}, 2n)]_{\#A^1}\ar@{=}[d]\\
[X, \BSp_{2n}]_{\#A^1}\ar[r]^-{b_n} & [X, \K(\KMW_{2n}, 2n)]_{\#A^1}
}$$ commutes such that the map 
$$[X, \BSp_{2n}]_{\#A^1}\xrightarrow{b_n} [X, \K(\KMW_{2n}, 2n)]_{\#A^1}$$
is taking the $n^{th}$ Borel class $b_n(\xi)$ which agrees with the Euler class $e_{2n}(P)$ (see~\cite[page 14858 (40)]{PD22}). Note that we have $[X, \K(\KMW_{2n}, 2n)]\cong \coh^{2n}(X,\KMW_{2n})$ which is isomorphic to $\widetilde{\CH}^{2n}(X)$.

The Euler class $e_{2n}(P)=0$ in  $\widetilde{\CH}^{2n}(X)$ \ie $b_n(\xi)=*$ implies that  the sequence of pointed sets $$ [X, E_{2n-1}]_{\#A^1} \xrightarrow{(q_{2n-1})_*} [X, E_{2n-2}]_{\#A^1} \xrightarrow{b_n} [X, \K(\KMW_{2n}, 2n)]_{\#A^1}$$ is exact. Hence there is $a\in [X, E_{2n-1}]_{\#A^1}$ such that $(q_{2n-1})_*(a)=\xi.$ As observed above the map 
$$[X, E_{2n}]_{\#A^1} \xrightarrow{(q_{2n})_*} [X, E_{2n-1}]_{\#A^1} $$
is surjective hence there is $a'\in [X, E_{2n}]$ such that $(q_{2n})_*(a')=a$. Furthermore, since the $\dim X=2n+1$ the maps $$ [X, \BSp_{2n-2}]\to [X, E_{r}]_{\#A^1}\to [X, E_{2n}]_{\#A^1}$$ are surjective for $r\geq 2n$. Hence the preimage $\xi'\in [X, \BSp_{2n-2}]_{\#A^1}$ of $a'$ is the desired lift of the map $\xi\in [X, \BSp_{2n}]_{\#A^1}.$

The remaining assertion follows in view of~\Cref{Vanishing:ABH}.
\end{enumerate}
\end{proof}

\begin{remark}
Theorem \ref{thm:splitting}(1) can be deduced from \cite[Theorem 8.14, Remark 8.15]{Mor12} and \cite[Proposition 4.1]{Bass69}. However, our proof is short, independent and can be seen as an application of affine representability result of Asok-Hoyois-Wendt \cite{AHW18}. A similar remark can be made for  Theorem \ref{thm:splitting}(2), in view of \cite{ABH23} and   \cite[Proposition 4.1]{Bass69}.
\end{remark}

 \section{Comparison of Euler class group and Chow group}\label{Euler-Chow}

 In this section, we study the map between $(d-1)-$th Euler class group $\E^{d-1}(A)$ and the $(d-1)-$the Chow group $\CH^{d-1}(X)$ where $X=\Spec A$ is a smooth affine scheme of Krull dimension $d$ over a field $k$.
For the definition and properties of the Chow-group, we refer to~\cite{Fulton}. We recall the definition of Euler class group from \cite{BR02}. For more recent account, see~\cite[Definition 3.1.5]{AF22}.

\begin{definition} (\cite[Section~4]{BR02})
Let $A$ be a commutative Noetherian ring of dimension $d$.
Assume $2n\geq d+3$.
 Now we recall the definition of Euler class group $\E^{n}(A)$. Let $I$ be an ideal  of $A$ of height $n$ such that $I/I^2$ is generated by $n$ elements as an $A/I-$module. Let $\alpha, \beta : (A/I)^{n} \twoheadrightarrow I/I^2$ be two surjections. We say 
$\alpha$ and $\beta$ are \emph{related} if there exists $\sigma \in E_{n}(A/I)$ such that $\alpha\circ \sigma = \beta$. Here $E_{n}(R)$ is the subgroup of $n\x n-$matrices generated by elementary matrices of determinant 1 with entries in a commutative ring $R$. One checks that this relation is an equivalence relation on the set of all surjections from $(A/I)^{n}$ to $I/I^2$. An equivalence class of surjections
from $(A/I)^{n} \twoheadrightarrow I/I^2$ is called a {\it local orientaion}. A local orientation $\omega_I$ of $I$ is called a {\it global orientation} if a surjection (hence all the surjections in the equivalence class by \cite[Section~4]{BR02}) in the class of $\omega_I$ can be lifted to a surjection from $A^{n} \twoheadrightarrow I$. 

Let $G$ be the free abelian group on the set $B$ of equivalence classes of pairs $(\mathcal{I},\omega_{\mathcal{I}})$, where $\mathcal{I}\subset R$
is an ideal of height $n$ with the property that $\Spec(A/\mathcal{I})$ is connected, $\mu(\mathcal{I}/\mathcal{I}^2)=n$ and $\omega_{\mathcal{I}}: (A/\mathcal{I})^{n}\twoheadrightarrow \mathcal{I}/\mathcal{I}^2$
is a local orientation. Now let $I$ be any ideal of height $n$ such that $I/I^2$ is generated by $n$ elements as an $A/I-$module. Then $I$ has a unique decomposition, $I=\mathcal{I}_1 \cap \ldots \cap \mathcal{I}_k$, where each of $\Spec(A/\mathcal{I}_i)$ is connected, $\mathcal{I}_i$
are pairwise maximal, and $ht(\mathcal{I}_i)=n$. If $\omega_I$ is a local orientation, then it naturally induces local orientation 
$\omega_{\mathcal{I}_i}: (A/\mathcal{I}_i)^{n} \twoheadrightarrow \mathcal{I}_i/\mathcal{I}_i^{2}$ for each $i,$ $1\leq i \leq k$. Thus
by $(I, \omega_I)$, we mean $\Sigma (\mathcal{I}_i, \omega_i) \in G$.

Let $H$ be the subgroup of $G$ generated by the subset $S$  of $B$ given by  $(I, \omega_I)$ in $G$ such that $\omega_I$ is a global orientation. We define the $n-$th Euler class group of $R$ as 
$\E^{n}(A):=G/H$.

\end{definition}

In the following remark, we compare Euler class groups defined by Bhatwadekar--Sridharan \cite{BR02}, Das \cite{MD21} and Asok--Fasel \cite{AF22}.

\begin{remark} \label{remk-1}In \cite{AF22}, Asok--Fasel defined the Euler class group $\E^n(A)$ slightly different way. However it is observed by Asok and Fasel \cite[Remark 3.16]{AF22} that they are same.

If $A$ is an affine algebra of dimension $d$ over $\overline{\mathbb{F}}_p$ where ${\mathbb{F}_p}$ is a finite field with $p$ elements, then Das \cite[Section ~4]{MD21}
defined $\E^{d-1}(A)$. We denote this by $\E^{d-1}_{MD}(A)$.
It is observed by Das \cite{MD21} that for $d\geq 5$ the natural map
$\E^{d-1}(A) \rightarrow \E^{d-1}_{MD}(A)$ is an isomorphism.

In fact, for $d=4$, the same arguments also proves that the natural map $\E^3(A)\rightarrow \E^3_{MD}(A)$ is an isomorphism as $E_3(A/I)=\SL_3(A/I)$ where $A$ is an affine algebra over $\overline{\mathbb{F}}_p$ and $I$ is an ideal in $A$ of height $3$.

Therefore, for $d\geq 4$, we shall write $\E^{d-1}(A)$ for $\E^{d-1}(A) \cong \E^{d-1}_{MD}(A)$.
\end{remark}

We recall the map 
\begin{equation} \label{eqn:EtoCH}
 \psi: \E^{d-1}(A)\to \CH^{d-1}(X): (I,w_I) \mapsto (I), 
\end{equation}
 where $(I)$ denotes the fundamental cycle associated to the closed subscheme $\Spec A/I$ in $\Spec A$.

 If $A$ is an affine algebra over $\overline{\mathbb{F}}_p$, then in \cite[Definition 4.1]{MD21}, a well-defined Euler class
 $e_{d-1}(P)$ is defined in $\E^{d-1}(A)$ for a projective $A-$module $P$ of rank $d-1$ with $\wedge^{d-1}P \cong A$. We note that
 the above map $\phi$ take Euler class $e_{d-1}(P)$ to Chern class $c_{d-1}(P)$ in $\CH^{d-1}(A)$.

 \begin{proposition}
\label{prop1}
 Let $X$ be a smooth affine variety of dimension $d$ over an infinite field $k$. For $d\geq 4$,  we have the Hurewicz map $$H: [X, Q_{2d-2}]_{\A^1}\to \~{\CH}^{d-1}(X)=\coh^{d-1}(X, \KMW_{d-1})=[X, \K(\piA_{d-1}(Q_{2d-2}), d-1)]_{\A^1}$$ Then
\begin{enumerate}
    \item[$(1)$] $H$ is surjective.
    \item[$(2)$] the cohomology group $\coh^d(X, \piA_d(Q_{2d-2}))$ surjects onto the kernel of the Hurewicz map $H$.
    \\
    In particular if $\coh^d(X, \piA_d(Q_{2d-2}))=0$ then $H$ is injective.
\end{enumerate}
 \end{proposition}
 
\begin{proof}
The proof uses similar arguments as in  ~\cite[Proposition 6.2]{AF14}. However we provide the details for the reader's convenience.
 We recall the notations from ~\cite[Proposition 1.2.6]{AF22}. Consider the Moore-Postnikov tower $\{E_{\Dot}\}$ for the morphism $Q_{2d-2} \rightarrow *$. To simplify the notation, for $j\geq 1$, we denote $\pi_{j}:=\piA_{j}(Q_{2d-2})$ and then note $\pi_{d-1}:=\piA_{d-1}(Q_{2d-2})=\KMW_{d-1}$ (see~\Cref{quadric-hom-type}).
Given $\alpha: X\to  \K(\piA_{d-1}(Q_{2d-2}), d-1) $, consider the following diagram 
$$\xymatrix@1{  &   &   E_d \ar[d]& \\ X\ar[r]^(.3){\alpha} & \K(\pi_{d-1}, d-1)\ar[r]^(.6){\gamma}  & E_{d-1}\ar[r]^{k_{d}}  & \K(\pi_{d}, d+1)\\}$$
We have $\beta:=\gamma\circ \alpha: X\to \K(\pi_{d-1}, d-1)\to E_{d-1}$ lifts to $E_d$ if and only if $\beta^*(k_d)=0$ in $\coh^{d+1}(X, \pi_d)$. Furthermore the set of lifts is parametrized by quotients of $\coh^{d}(X, \piA_{d}(Q_{2d-2})).$
Now since $X$ has Krull dimension $d$, we get  $\coh^{d+1}(X, \pi_d)=0$. So $\beta$ lifts to $E_d$. Further the lift is unique, by assumption in (2).
{And the set of such possible lifts corresponds to elements of $\coh^d(X, \pi_d)$ which is also 0 by \cite[Prop. 1.1.5]{AF22}, hence the lift is unique.} 
Continuing this way for $i\geq 1$, $\beta$ lifts to $E_{d-1+i}$ if $\coh^{d+i}(X, \pi_{d-1+i})=0$ (again because Krull dimension of $X=d$).  And the set of lifts is parametrized by quotients of $\coh^{d-1+i}(X,\pi_{d-1+i})$ which is zero for $i=1$ by assumption in (2) and for $i>1$ because Krull dimension of $X$ is $d$ and $d-1+i>d.$
{and the lift is unique if $\coh^{d-1+i}(X, \pi_{d-1+i})=0$ (which follows again from Krull dimension of $X$).}
Thus the map $$H: [X, Q_{2d-2}]_{\A^1}\to \~{\CH}^{d-1}(X)$$ is surjective and under the assumption in (2) the map $H$ is injective.
\end{proof}

\begin{remark}
In the above proposition, under the hypothesis $d\geq 4$ the set 
$[X, Q_{2d-2}]_{\A^1}$ has a group structure (see \cite[Theorem 1.3.4]{AF22}).

At this point, we do not know whether the injectivity of $H$ implies
$\coh^d(X, \piA_d(Q_{2d-2}))=0$.
\end{remark}

\begin{proposition}\label{prop2}
Let $k$ be an infinite field with $cd_2(k)=0$, (i.e. the \'etale cohomological dimension of the field $k$ at the prime 2 is zero). Let $X$  a smooth affine variety of dimension $d\geq 2$ over $k$.
Then the natural  map $$\~{\CH}^{d-1}(X)\to \CH^{d-1}(X) $$ is an isomorphism.
\end{proposition}
\begin{proof}
We consider the map $$\~{\CH}^{d-1}(X)=\coh^{d-1}(X, \KMW_{d-1})\to \CH^{d-1}(X)=\coh^{d-1}(X, \KM_{d-1}) $$
We have a short exact sequence of Zariski/Nisnevich sheaves 
$$0\to \&I^{d}\to \KMW_{d-1}\to \KM_{d-1}\to 0$$
which gives a long exact sequence of cohomlogies 
$$\cdots\to \coh^{d-1}(X, \&I^{d})\to \coh^{d-1}(X, \KMW_{d-1})\to \coh^{d-1}(X, \KM_{d-1})\to \coh^{d}(X, \&I^{d})\to\cdots$$
Now by ~\cite[Proposition 5.2]{AF14}, (for $cd_2(k)=r=0$), $\coh^{d-1}(X, \&I^{d})=0 $ and $\coh^{d}(X, \&I^{d})=0$.

Thus, the map $$\coh^{d-1}(X, \KMW_{d-1})\to \coh^{d-1}(X, \KM_{d-1})$$ is an isomorphism, hence the result.
\end{proof}

\begin{theorem} \label{Main-1}
Let $k$ be an algebraically closed field
and $A$ be a smooth affine algebra of dimension $d\geq 4$ over $k$. Let $X=\Spec(A)$.
Consider the map $$\psi: \E^{d-1}(A)\to \CH^{d-1}(X).$$ Then 
\begin{enumerate}
    \item the map $\psi$ is surjective. 
    \item further under the assumption $\coh^d(X, \piA_d(Q_{2d-2}))=0$, the map $\psi$ is isomorphism.
\end{enumerate}
\end{theorem}
\begin{proof}(1)
 We first claim that the following diagram 
$$\xymatrix{
\E^{d-1}(A)\ar[r]^-{\simeq}_-{s}\ar[rd]\ar[rdd]_{\psi} & [X, Q_{2d-2}]_{\A^1}\ar@{->>}[d]^{H}\\
        & \~{\CH}^{d-1}(X)\ar@{->}[d]^{\simeq ~by~ Proposition~\ref{prop2}}\\
        & \CH^{d-1}(X)
}$$ commutes, where $s$ denotes the Segre class map, $H$ the Hurewicz map and $\psi$ is defined in (\ref{eqn:EtoCH}). 
Before justifying the commutativity of the above diagram, from \cite[Definition 3.1.3.]{AF22} let us recall the Segre class morphism
$$s: \E^{d-1}(A)\to \pi_0(Sing_*^{\A^1}(Q_{2d-2})(A))\simeq [X, Q_{2d-2}]_{\A^1}.$$ 
Note that Asok-Fasel \cite{AF22} rewrote the definition of Euler class group slightly differently, which is  same as Bhatwadekar--Raja Sridharan's Euler class group (see Remark 3.1.6 of \cite{AF22}). With this in mind, for $(I, \w_I)\in \E^{d-1}(A)$, there exist $a_1,\ldots, a_{d-1}, t \in I$ and $b_1,\ldots,b_{d-1} \in A$ such that $I=<\_a, t>$ and $(\_a,\_b,t) \in Q_{2d-2}(A)$ where $\_a=(a_1,\ldots,a_{d-1}),\, \_b=(b_1,\ldots,b_{d-1})$ (see \cite[Theorem 3.1.2.]{AF22}). The Segre class map is an isomorphism which follows combining \cite[Proposition 3.1.9]{AF22} and \cite[Theorem 4 (3)]{Asok22}. We remark that \cite[Proposition 3.1.9]{AF22} assumes $char(k)\neq 2$ and \cite[Theorem 4 (3)]{Asok22} removes the characteristic hypothesis assuming $k$ infinite. Note in our case $k$ is algebraically closed, hence infinite.

By Proposition (\ref{prop1}), the Hurewicz map $H$ is surjective. 

For commutativity of the above diagram, let $[(I, \w)]\in \E^{d-1}(A)$. Then $\psi(I, \w_I)=(I)$, the fundamental class associated to the closed subscheme $\Spec A/I$ in $X$ and
$s(I, \w)=([\_a, \_b, t])=f$ (say)  where $I=<\_a, t>$. 
Then $$f: X=\Spec A\to Q_{2d-2}=\Spec \dfrac{k[\_x, \_y, z]}{\sum x_iy_i-z(1-z)}$$ is the morphism, defined as $$x_i\mapsto a_i, \ y_i\mapsto b_i, \ z\mapsto t.$$  
Under the Hurewicz map $$H: [X, Q_{2d-2}]_{\A^1}\to \~{\CH}^{d-1}(X)=\coh^{d-1}(X, \KMW_{d-1})$$ 
$$f\mapsto f^*(\alpha)$$ where $\alpha\in \coh^{d-1}(X, \KMW_{d-1})$.
Now $\alpha$ corresponds to $<1>\otimes \-{x}_1\wedge\cdots \-{x}_{d-1}\in \KMW_0(\kappa(w))\otimes \det \fr m_w/\fr m_w^2$ where $w$ is the generic point of the codimension $d-1$ irreducible closed subscheme $W$ of $Q_{2d-2}$ given by
the equation $$x_1=\cdots= x_{d-1}=z=0.$$
We have the pullback map $$f^*: \coh^{d-1}(Q_{2d-2}, \KMW_{d-1})\to \coh^{d-1}(X, \KMW_{d-1})$$
such that $$f^*(\alpha)=f^*(<1>\otimes \-x_1\wedge\cdots \wedge \-x_{d-1})= \sum_{u} l_u <1>\otimes \-a_1\wedge\cdots \wedge \-a_{d-1} $$ where the sum varies over the generic points $u$ of $f^{-1}(W)$ and $l_u$ is the length of the Artinian local ring $\sO_{f^{-1}(W), u}$.

Under the map $$\coh^{d-1}(X, \KMW_{d-1})\to \coh^{d-1}(X, \KM_{d-1})$$ the element  
$f^*(\alpha)$ maps to $\sum_{u} l_u[u]$ which is precisely the fundamental class associated to the closed subscheme $f^{-1}(W)$, but we note that $f^{-1}(W)=\Spec A/I$.  This completes the proof of commutativity.

By~\cite[Theorem 1(4)]{AF22}, we have  $\E^{d-1}(A)\to [X, Q_{2d-2}]_{\A^1}$ is an isomorphism.
Hence the composition map $\psi$ is surjective. 

(2) This follows directly from Proposition~\ref{prop1} (2), as the assumption implies that $H$ is an isomorphism. Hence the composition $\psi$ is an isomorphism. 
\end{proof}

As a consequence of combining the above theorem with the vanishing results from section~\ref{vanishing-sec}, we prove the following result. 
\begin{corollary}\label{cor:iso} Let $k$ be an algebraically closed field of characteristic $\neq 2$.
Let $X=\Spec A$ be a smooth affine $k-$scheme of dimension $d\geq 4$.

Consider the map $$\psi: \E^{d-1}(A)\to \CH^{d-1}(X).$$ Then the following statements hold. 

\begin{enumerate}
    \item[$(1)$] The map $\psi$ is always surjective.

 \item[$(2)$] If $d=4$, then the map $\psi$ is an isomorphism. 
 \item[$(3)$] If further characteristic of $k$ is $0$, then the map $\psi$ is an isomorphism. 
\item[$(4)$]  If $d\geq 5$, then the kernel of the map $\psi$ is $d!\cdot (d-1)!\cdot(d-2)!-$torsion.
         \item[$(5)$] If $d\geq 5$ and $d!\in k^{\x}$, then the kernel of the map $\psi$ is $(d-1)!\cdot(d-2)!-$torsion.
\item[$(6)$] Further for $d\geq 5$, if Conjecture \ref{AF-conj} holds, (which is the case when characteristic of $k=0)$  then the map $\psi$ is an isomorphism.
\end{enumerate}
\end{corollary}

\begin{proof}
The assertion (1) follows by combining Remark \ref{remk-1} and Theorem \ref{Main-1}.  Case (2) follows from Theorem \ref{Main-1} and Theorem~\ref{vanishing}~(3). Case (3) follows from Theorem~\ref{vanishing}~(4). Case (4) follows from Theorem~\ref{vanishing} (5).  Case (5) follows from Theorem~\ref{vanishing} (6).  Case (6) follows from Theorem~\ref{vanishing} (7).
\end{proof}

\end{document}